\documentclass[10pt, a4paper, reqno]{amsart}
\numberwithin{equation}{section}
\usepackage{amsmath,amssymb,amsthm}
\usepackage{xcolor,textcomp}
\usepackage{graphicx, epstopdf}
\usepackage{braket,amsfonts}
\usepackage{dsfont}
\usepackage{mathtools}
\usepackage{mwe}
\usepackage[top=2.9cm,bottom=2.75cm,left=2.75cm,right=2.75cm]{geometry}
\usepackage{hyperref}
\hypersetup{colorlinks=true,linkcolor=red,citecolor=blue,filecolor=magenta,urlcolor=cyan}
\usepackage[capitalise]{cleveref}
\usepackage{comment}
\usepackage{tikz}
\usepackage{pgfplots}
\usepackage{stmaryrd}
\usepackage[mathscr]{euscript}
\usepackage[sort,nocompress]{cite}

\newcommand{\vphi}{\varphi}

\newcommand{\B}{\mathcal B}

\newcommand{\C}{\mathcal C}

\newcommand{\oo}{\mathcal{O}}

\newcommand{\y}{\widetilde{y}}

\newcommand{\Y}{\mathcal{Y}}

\newcommand{\G}{\mathcal{G}} 

\newtheorem{theorem}{Theorem}[section]
\newtheorem{lemma}[theorem]{Lemma}
\newtheorem{remark}[theorem]{Remark}

\newtheorem{proposition}[theorem]{Proposition}

\makeatletter
\@namedef{subjclassname@2020}{%
	\textup{2020} Mathematics Subject Classification}
\makeatother

\title[Insensitizing control problem for a KdV-KdV system]{ Insensitizing control problem for the Hirota-Satsuma system of KdV-KdV type}
\author[K. Bhandari]{Kuntal Bhandari$\,^\dagger$}
\thanks{$^\dagger\,$Institute of Mathematics of the  Czech Academy of Sciences, \v{Z}itn\'a 25, 11567 Praha 1, Czech Republic;  bhandari@math.cas.cz.}


\date{\today}

\begin{document}
	\begin{abstract}
	This paper is concerned with the existence of insensitizing controls for a nonlinear coupled system of two Korteweg-de Vries (KdV) equations, typically known as the Hirota-Satsuma system.
	 The  idea is to look for controls such that some functional of the states (the so-called sentinel) is insensitive to the small perturbations of  initial data.  Since the system is coupled, 
	 we  consider a sentinel in which we observe both components of the system in a localized observation set. By some classical argument, the insensitizing problem is then reduced to a null-control problem for an extended system where the number of equations is doubled. We study the null-controllability for the linearized model associated  to that extended system by means of a suitable Carleman estimate  which is proved in this paper.    Finally, the local null-controllability  of the extended (nonlinear) system is obtained by applying the inverse mapping theorem, and this  implies   the required insensitizing property for the concerned model.
\end{abstract}

	
		\maketitle
	
\section{Introduction}\label{Sec-Introduction}

\subsection{Statement of the problem}\label{Sec-Problem-stat}
	In this article, we study an insensitizing control problem for the Hirota-Satsuma system of two coupled Korteweg-de Vries (KdV) equations. In $1981$, R. Hirota and J. Satsuma \cite{HIROTA-SATSUMA} have proposed a system of two coupled KdV equations, namely
	\begin{align*}
		\begin{cases}
				u_t -\frac{1}{2}u_{xxx} - 3uu_x + 6vv_x = 0, \\
			v_t + v_{xxx} + 3uv_x  =0,
		\end{cases}
	\end{align*} 
which describes the interactions of two long waves with different dispersion relations. They have presented a soliton solution of the above system, and shown that it has two- and three-soliton solutions  under a special connection between the dispersion relations of the two long waves;   we  refer \cite{HIROTA-SATSUMA} for more details. 

\smallskip
 
Let us  describe the problem on which  we are going to work in the present article. Let $T>0$ be  given finite time and $L>0$ be any finite length. Denote $Q_T:= (0,T)\times (0,L)$. We further assume two non-empty open sets $\omega\subset (0,L)$ and $\mathcal O \subset (0,L)$ verifying $\mathcal O \cap \omega \neq \emptyset$, and that there is a non-empty open set $\omega_0$ such that $\omega_0 \subset \subset \mathcal O \cap \omega$.

\smallskip 

We now consider the following coupled KdV system:
	\begin{align}\label{Control-system-main}
		\begin{cases}
			u_t -\frac{1}{2}u_{xxx} - 3uu_x + 6vv_x = h_1 \mathds{1}_\omega + {\xi}_1  &  \text{in } Q_T, \\
		v_t + v_{xxx} + 3uv_x  = h_2 \mathds{1}_\omega + \xi_2  &  \text{in } Q_T, \\
		u(t,0) = u(t,L) = u_x(t,0) = 0 & \text{for } t\in (0,T), \\
		v(t,0) = v(t,L) = v_x(t,L) =0 & \text{for } t\in (0,T) , \\
		u(0) = u_0 + \tau \widehat{u}_0 , \ \, v(0) = v_0 + \tau \widehat{v}_0  & \text{in } (0,L)  ,
		\end{cases}
	\end{align}
where  $u=u(t,x)$ and $v=v(t,x)$ are the state variables, $h_i=h_i(t,x)$, $i=1,2$ are localized control functions acting on the subset $\omega$, and $\xi_i=\xi_i(t,x)$, $i=1,2$, are given external source terms. In \eqref{Control-system-main}  the initial states $(u(0),v(0))$ are partially unknown in the following sense:

\smallskip 

\begin{itemize}
	\item[--]  $(u_0,v_0) \in [L^{2}(0,L)]^2$ are given, and 
	\item[--]  $(\widehat{u}_0 , \widehat{v}_0)\in [L^{2}(0,L)]^2$ are unknown which satisfy $\|\widehat{u}_0 \|_{L^2(0,L)} = \|\widehat{v}_0 \|_{L^2(0,L)} = 1$. They represent some {\em uncertainty} on the initial data. 
	
	\item[--]   $\tau\in\mathbb R$ is some unknown parameter which is small enough. 
\end{itemize}

Our aim is to study the insensitizing control problem for the system \eqref{Control-system-main}.     This  topic has been originally introduced by J.-L. Lions  \cite{Lio90}  which concerns with the existence of controls that make   a certain functional (depending on the state variables) insensible
with respect to small perturbations of the initial data.
   
In our case, we  consider the following functional (the so-called {\em sentinel})  defined  on the set of solutions to \eqref{Control-system-main},  given by 
\begin{align}\label{sentinel_functional}
	J_\tau(u,v) : =	\frac{1}{2}\iint_{(0,T)\times \mathcal O} |u|^2 + \frac{1}{2}  \iint_{(0,T)\times \mathcal O} |v|^2 ,
\end{align}
where $\mathcal O$ is the so-called {\em observation domain}.  Then, the insensitizing control problem associated to \eqref{Control-system-main} can be stated as follows.

\smallskip 

Let $(u_0, v_0)\in [L^2(0,L)]^2$ and $(\xi_1, \xi_2)\in [L^2(Q_T)]^2$ be given. We say that the control functions $(h_1,h_2)\in [L^2((0,T)\times \omega)]^2$ insensitize the sentinel functional $J_\tau$ given by  \eqref{sentinel_functional}, if  
\begin{equation}\label{eq:ins_satisfy}
	\left.\frac{\partial J_\tau(u,v)}{\partial \tau}\right|_{\tau=0}=0, \quad \forall \, (\widehat u_0,\widehat v_0)\in [L^{2}(0,L)]^2  \textnormal{ with } \|\widehat u_0 \|_{L^{2}(0,L)}=\|\widehat v_0\|_{L^{2}(0,L)} = 1.
\end{equation}

\smallskip 

Thus, it amounts to find a pair of control functions $(h_1,h_2)\in [L^2((0,T)\times \omega)]^2$ such that the uncertainty in the initial data does not affect the measurement of $J_\tau$. When the condition \eqref{eq:ins_satisfy} holds for a pair of controls $(h_1,h_2)$,  the sentinel $J_\tau$ is said to be {\em locally insensitive} to  the perturbations of  initial data. In other words, \eqref{eq:ins_satisfy} indicates that
the sentinel does not detect the variations
of the given initial data $(u_0,v_0)$  by  small unknown perturbation $(\tau \widehat u_0, \tau\widehat v_0)$ in the observation domain $\mathcal O$. 

\subsection{Bibliographic comments}

To begin with, the controllability of dispersive systems is gaining its popularity among several researchers;   the first internal controllability   results regarding the single KdV equations have been presented by Russell and Zhang  \cite{Russell-Zhang-1,Russell-Zhang-2} in the periodic domains.  We also cite the  work  \cite{Rosier-KdV-CPDE} by Rosier and Zhang related to this topic. In \cite{Glass-Guerrero}, Glass and Guerrero proved the exact internal controllability of the KdV equation when the control acts in a neighborhood of the left endpoint.  Later on, Capistrano-Filho, Pazoto and Rosier  \cite{Roberto-Rosier-internal-kdv} established a Carleman estimate yielding an observability inequality to conclude the internal controllability for the KdV equations.
More recently, the internal null-controllability  of a generalized Hirota-Satsuma system (which is a coupled system of three KdV equations with a first order coupling)  has been proved by Carre\~{n}o, Cerpa and Cr\'{e}peau in \cite{Carreno-Hirota-Satsuma}.  

On the other hand, the  {\em boundary controllability} for KdV equations has been intensively investigated by several renowned researchers, see for instance,  \cite{Cerpa-single-Kdv,Cerpa-BOUNDARY-single-KDV,Zhang-SICON, Cerpa-Zhang-Boundary,GLASS-Guerreo-System-Control,Coron-nonlinear-KdV,Rosier-single-Kdv}.   Most of those works were concerned with the system 
\begin{align}\label{sys-boundary}
	\begin{cases}
		u_t + u_x +u_{xxx} + uu_x = 0    &\text{in } (0,T)\times (0,L) , \\
		u(t,0)= p_1(t), \  u(t,L) = p_2(t),  \ u_x(t,L) = p_3(t) & \text{in } (0,T) , 
	\end{cases}
\end{align}
where $p_1,p_2, p_3$ are control inputs.  In particular, 
Rosier \cite{Rosier-single-Kdv} proved that the linearized model of \eqref{sys-boundary} is exactly controllable with a control $p_3\in L^2(0,T)$ (and $p_1=p_2=0$) if and only if $L$ does not belong to the following countable set of critical length 
\begin{align*}
	\mathcal N = \left\{ \frac{2\pi}{\sqrt{3}} \sqrt{k^2+ kl + l^2} \, : \, k,l \in \mathbb N^*   \right\} .
\end{align*} 
He has also shown that when the linearized equation is controllable, the same holds true for the nonlinear equation.  But the converse is not necessarily true, 
as it has been proved in \cite{Cerpa-single-Kdv,Cerpa-BOUNDARY-single-KDV,Coron-nonlinear-KdV} that the nonlinear KdV equation is still
controllable even if $L\in \mathcal N$. Finally, we refer \cite{Roberto-Boundary_KDV,Micu,Cerpa-Note_on-Micu} where the boundary controllability of some system of KdV equations have been addressed.  

\medskip 

Let us talk about the insensitizing control problems for pdes.  We mention that the pioneer results concerning the existence of insensitizing controls were obtained by de Teresa in  \cite{deT00}, and by Bodart et al.  in \cite{BF95} for the linear and semilinear heat equations. 

Since then, many works have been devoted to study the insensitizing problem from different perspectives. The authors in \cite{BGBP04,BGP04b,BGP04c} study such problems for linear and semilinear heat equations with different types of nonlinearities and/or boundary conditions.  In the direction of wave equation,  Alabau-Boussouira \cite{Alabau2014insensitizing} proved the existence of exact insensitizing controls for the scalar wave equation.  

Guerrero  \cite{Sergio} has considered an insensitizing control problem for the linear parabolic equation where the  {\em sentinel functional} is dependent on the gradient of the solution.  Later,  he  has studied such control problem for the Stokes equation \cite{Guerrero2007Stokes} where the sentinel  is  taken in terms of  the {\em curl of solution}. 
In the context of insensitizing problems for Navier-Stokes equations  we quote the works \cite{Gue13,CG14},  and for the Boussinesq system  we cite \cite{CGG15,Cn17}.  We also refer \cite{CCC16} where the insensitizing control problem for a phase field system  has been explored.   Furthermore, a numerical study for the insensitizing property of semilinear parabolic equations has been pursued in \cite{BHSdeT19}. 

It is worth mentioning that insensitizing problems  for the fourth-order parabolic equations have been treated  in \cite{Kas20} and with respect to shape variations in \cite{LPS19,ELP20}. We also bring up a recent work \cite{T20} where the authors studied the  insensitizing property for the fourth-order dispersive nonlinear Schr\"odinger equation with cubic nonlinearity. 
  Last but not the least, we address a very recent work \cite{Bhandari2023insensitizing} where  the insensitizing control problems for the stabilized Kuramoto-Sivashinsky system have been  analyzed. 

 
 
 
 \smallskip 
 
 In the current work, we investigate the insensitizing property of the so-called Hirota-Satsuma model, which is basically a nonlinear coupled KdV system as mentioned earlier, and to the best of our knowledge, this problem has not been addressed in the literature.    
 

\subsection{Main results}
As mentioned earlier, our goal  is  to prove the existence of control functions $(h_1, h_2)$ which insensitize the functional $J_\tau$  given by \eqref{sentinel_functional}.
In this regard, our  control result is the following.
\begin{theorem}\label{thm:main}
	Assume that $\oo \cap \omega \neq \emptyset$ and $u_0\equiv v_0\equiv 0$.  Then, there exist constants $C>0$ and $\delta>0$  
	such that for any $(\xi_1,\xi_2)\in [L^2(Q_T)]^2$ satisfying
	\begin{equation}\label{eq:cond_rho-2}
		\|e^{C/t}(\xi_1, \xi_2)\|_{[L^2(Q_T)]^2} \leq \delta ,
	\end{equation}
	one can prove the existence of  control functions $(h_1,h_2)\in [L^2((0,T)\times \omega)]^2$ which insensitize the functional $J_\tau$ in the sense of \eqref{eq:ins_satisfy}. 
\end{theorem}

To prove the above theorem, we shall equivalently establish the result given by \Cref{thm:main_extended} below.  In fact, adapting the  arguments in \cite[Proposition 1]{BF95} or \cite[Appendix]{deTZ09}, it can be proved that the insensitivity condition \eqref{eq:ins_satisfy} is equivalent to a null-control problem for an extended  system, which is in our case given by 
	\begin{align}\label{sys_equiv_1}
	&\begin{cases}
		u_t -\frac{1}{2}u_{xxx} - 3uu_x + 6vv_x = h_1 \mathds{1}_\omega + \xi_1  &  \text{in } Q_T, \\
		v_t + v_{xxx} + 3uv_x  = h_2 \mathds{1}_\omega + \xi_2  &  \text{in } Q_T, \\
			u(t,0) = u(t,L) = u_x(t,0) = 0 & \text{for } t\in (0,T), \\
		v(t,0) = v(t,L) = v_x(t,L) =0 & \text{for } t\in (0,T) , \\
		u(0) = u_0, \ \ v(0) = v_0 & \text{in } (0,L) ,
		\end{cases} \\
	\label{sys_equiv_2}
	&\begin{cases} 	
		-p_t + \frac{1}{2}p_{xxx} - 3p u_x + 3 q v_x = u \mathds{1}_{\mathcal O}  &  \text{in } Q_T, \\
		-q_t - q_{xxx} + 6pv_x =  v \mathds{1}_{\mathcal O}  &  \text{in } Q_T, \\
			p(t,0) = p(t,L) = p_x(t,L) = 0 & \text{for } t\in (0,T), \\
		q(t,0) = q(t,L) = q_x(t,0) =0 & \text{for } t\in (0,T) , \\
		p(T)=0, \ \ q(T)=0 & \text{in } (0,L),
		\end{cases}
	\end{align}
and we have the following result. 
\begin{proposition}\label{Prop-primal}
A pair of  control functions $(h_1,h_2)\in [L^2((0,T)\times \omega)]^2$ verifies the insensitivity condition \eqref{eq:ins_satisfy} for the sentinel \eqref{sentinel_functional} if and only if the associated solution to \eqref{sys_equiv_1}--\eqref{sys_equiv_2} satisfies
\begin{equation}\label{eq:null_cascade}
		(p(0),q(0))=(0,0) \quad \textnormal{in } (0,L). 
\end{equation}
\end{proposition}

In what follows, we only focus on studying the controllability properties for the $4\times 4$ forward-backward system \eqref{sys_equiv_1}--\eqref{sys_equiv_2}. Indeed, we prove the following theorem which is the main result of our paper. 
\begin{theorem}[Local null-controllability of the extended system]\label{thm:main_extended}
	Assume that $\oo \cap \omega \neq \emptyset$ and $u_0\equiv v_0\equiv 0$.  Then, there exist constants $C>0$ and $\delta>0$  
	such that for any $(\xi_1,\xi_2)\in [L^2(Q_T)]^2$ verifying 
	\begin{equation}\label{eq:cond_rho}
		\|e^{C/t}(\xi_1, \xi_2)\|_{[L^2(Q_T)]^2} \leq \delta ,
	\end{equation}
	there exist control functions $(h_1,h_2)\in [L^2((0,T)\times \omega)]^2$ such that the solution $(u,v,p,q)$ to \eqref{sys_equiv_1}--\eqref{sys_equiv_2} satisfies $p(0)=q(0)=0$ in $(0,L)$.   
\end{theorem}

As usual,   to prove  \Cref{thm:main_extended}, one needs to  first establish a global null-controllability result for the linearized (around zero) model associated to \eqref{sys_equiv_1}--\eqref{sys_equiv_2}. More precisely, we consider the following system:
		\begin{align}
			\label{sys_linear-1}
		&\begin{cases}
			u_t -\frac{1}{2}u_{xxx}  = h_1 \mathds{1}_\omega + f_1  &  \text{in } Q_T, \\
			v_t + v_{xxx}   = h_2 \mathds{1}_\omega + f_2  &  \text{in } Q_T, 
			\\
			u(t,0) = u(t,L) = u_x(t,0) = 0 & \text{for } t\in (0,T), \\
			v(t,0) = v(t,L) = v_x(t,L) =0 & \text{for } t\in (0,T) , \\
			u(0) = u_0, \ \ v(0) = v_0 & \text{in } (0,L) ,
		\end{cases} \\
	\label{sys_linear-2}
		&\begin{cases} 	
			-p_t + \frac{1}{2}p_{xxx} =  u \mathds{1}_{\mathcal O} + f_3  &  \text{in } Q_T, \\
			-q_t - q_{xxx}  =  v \mathds{1}_{\mathcal O}  + f_4 &  \text{in } Q_T, \\
			p(t,0) = p(t,L) = p_x(t,L) = 0 & \text{for } t\in (0,T), \\
			q(t,0) = q(t,L) = q_x(t,0) =0 & \text{for } t\in (0,T) , \\
			p(T)=0, \ \ q(T)=0 & \text{in } (0,L).
		\end{cases}
	\end{align}
with given right hand sides $(f_1, f_2, f_3,f_4)$ from some certain space (specified later). 

Note that, the controls $(h_1 \mathds{1}_\omega,h_2\mathds{1}_\omega)$  act directly in the equations of  $(u,v)$  while the equations of $(p,q)$ are indirectly controlled via the couplings $(u\mathds{1}_\oo, v\mathds{1}_\oo)$. At this point, one can observe that  the condition $\oo\cap \omega\neq \emptyset$   is necessary to obtain the required 
null-controllability result for the extended system \eqref{sys_linear-1}--\eqref{sys_linear-2}, in other words, the insensitizing  property for the main system \eqref{Control-system-main}.

	
	As we know, proving the null-controllability of \eqref{sys_linear-1}--\eqref{sys_linear-2} is equivalent to determine the {\em observability property} of its adjoint system, which is   given by  
		\begin{align}\label{adj-extended-1}
		&\begin{cases}
		-\eta_t +\frac{1}{2}\eta_{xxx}  =  \zeta \mathds{1}_{\mathcal O} +  g_1  &  \text{in } Q_T, \\
			-\psi_t - \psi_{xxx}   =  \theta \mathds{1}_{\mathcal O}  +  g_2  &  \text{in } Q_T, 
			\\
			\eta(t,0) = \eta(t,L) = \eta_x(t,L) = 0 & \text{for } t\in (0,T), \\
			\psi(t,0) = \psi(t,L) = \psi_x(t,0) =0 & \text{for } t\in (0,T) , \\
			\eta(T) = 0, \ \ \psi(T) = 0 & \text{in } (0,L)  ,
		\end{cases} \\
	\label{adj-extended-2}
		&\begin{cases} 	
			\zeta_t - \frac{1}{2}\zeta_{xxx} = g_3  &  \text{in } Q_T, \\
			\theta_t + \theta_{xxx}  = g_4 &  \text{in } Q_T, \\
			\zeta(t,0) = \zeta(t,L) = \zeta_x(t,0) = 0 & \text{for } t\in (0,T), \\
			\theta(t,0) = \theta(t,L) = \theta_x(t,L) =0 & \text{for } t\in (0,T) , \\
			\zeta(0)=\zeta_0, \ \ \theta(0)=\theta_0 & \text{in } (0,L) . 
		\end{cases}
	\end{align}
with given $(\zeta_0, \theta_0)\in [L^2(0,L)]^2$ and source terms $(g_1,g_2,g_3,g_4)$ from some suitable space, specified later.  



 The problem amounts  to establish a suitable Carleman estimate satisfied by the state variables
  of the $4\times 4$ system \eqref{adj-extended-1}--\eqref{adj-extended-2} with only two observation terms, namely $\eta$ and $\psi$.  We shall discuss it at length in Section \ref{Sec-Carleman}.


\subsection*{Paper Organization} The paper is organized as follows.
 Section \ref{Sec-well} contains the well-posedness results of  the underlying  coupled KdV systems. Then, in Section \ref{Sec-Carleman} we prove a suitable Carleman estimate for the $4\times 4$ adjoint system \eqref{adj-extended-1}--\eqref{adj-extended-2}. 
	The Carleman estimate (see \Cref{Carleman-main}) then yields  an observability inequality which is obtained  in Subsection \ref{Section-obser}, precisely \Cref{Prop-obs-ineq}. In Subsection \ref{Section-null-liniear}, we establish the null-controllability of the linearized model \eqref{sys_linear-1}--\eqref{sys_linear-2}, thanks to the  appropriate observability inequality. Afterthat,  in Section \ref{Section-null-nonlinear}, 
	 we prove the local null-controllability of the system \eqref{sys_equiv_1}--\eqref{sys_equiv_2}, which is precisely the proof of \Cref{thm:main_extended}.  Finally, we conclude our paper by providing several remarks in Section \ref{Section-conclusion}.
	 


\subsection*{Notation}  
Throughout the paper, $C>0$ denotes a generic constant that may vary line to line and depend on $\oo$, $\omega$, $L$ and $T$. 


 

\section{Well-posedness results}\label{Sec-well}

\subsection{Functional setting and well-posedness of single KdV equation}
We start by introducing the following functional spaces:
\begin{equation}\label{spaces-X} 
\begin{aligned}
	& X_0 := L^2(0,T; H^{-2}(0,L)) , \quad  \ X_1 := L^2(0,T; H^2_0(0,L)) , \\
	& \widetilde X_0 := L^1(0,T; H^{-1}(0,L)) , \quad \  \widetilde X_1 := L^1(0,T; H^3(0,L) \cap  H^2_0(0,L)) , \\
\end{aligned}
\end{equation}
and 
\begin{equation}\label{spaces-Y} 
	\begin{aligned}
&Y_0 := L^2(0,T; L^{2}(0,L)) \cap \C^0([0,T]; H^{-1}(0,L)), \\ 
&Y_1:= L^2(0,T; H^{4}(0,L)) \cap \C^0([0,T]; H^3(0,L) ) ,
	\end{aligned}
\end{equation}
which are equipped with their usual norms. For each $\mu \in [0,1]$, we further define the interpolation spaces (see for instance \cite{Bergh-Interpolation,Lions-Magenes}):
\begin{equation}\label{interpolation-spaces} 
	X_\mu : = (X_0, X_1)_{[\mu]} , \quad  \  	\widetilde X_\mu : = (\widetilde X_0, \widetilde X_1)_{[\mu]}  , \quad \
   Y_{\mu} : = (Y_0, Y_1)_{[\mu]}   .
\end{equation}
In particular, we have
\begin{align}
	\label{spaces-1/4}
&\begin{dcases}
 X_{\frac{1}{4}} = L^2(0,T; H^{-1}(0,L)) , \quad \widetilde X_{\frac{1}{4}} = L^1(0,T; L^2(0,L)) , \\
 Y_{\frac{1}{4}}	 =  L^2(0,T; H^1(0,L)) \cap \C^0([0,T]; L^2(0,L) ) ,
\end{dcases} \\
\label{spaces-1/2}
&\begin{dcases}
	 X_{\frac{1}{2}} = L^2(0,T; L^2(0,L)) , \quad \widetilde X_{\frac{1}{2}} = L^1(0,T; H^1_0(0,L)) , \\
	Y_{\frac{1}{2}}	 =  L^2(0,T; H^2(0,L)) \cap \C^0([0,T]; H^1(0,L) ) .
\end{dcases} 
\end{align}
Also, one can observe   that for any $\nu \in (0,1]$,
\begin{align} 
\label{spaces-nu/4}
\begin{dcases}
	 X_{\frac{1}{2} + \frac{\nu}{4}} = L^2(0,T; H^\nu(0,L)) , \quad \widetilde X_{\frac{1}{2} + \frac{\nu}{4}} = L^1(0,T; H^{1+\nu}(0,L)) , \\
	Y_{\frac{1}{2} + \frac{\nu}{4}}	 =  L^2(0,T; H^{2+\nu}(0,L)) \cap \C^0([0,T]; H^{1+\nu}(0,L) ) . 
\end{dcases}
\end{align}

Let us now consider the single KdV equation given by 
\begin{align}\label{single-KdV}
	\begin{cases} 
	y_t \pm  y_{xxx} = f & \text{in } Q_T, \\
	y(t,0) = y(t,L) = y_x(t,L) = 0 & \text{for } (0,T) , \\
	y(0,x) = y_0(x) & \text{in } (0,L) .
	\end{cases}
\end{align}
with given source term $f$ and initial data $y_0$. 

We recall the following known results for \eqref{single-KdV}.  

\begin{lemma}[{\cite[Section 2.2.2]{Glass-Guerrero}}]\label{Prop-1-well}
	For given $y_0\in L^2(0,L)$ and $f\in F$ with   $F=X_{\frac{1}{4}}$ or $\widetilde X_{\frac{1}{4}}$, the system \eqref{single-KdV} admits a unique solution $y\in Y_{\frac{1}{4}}$, and in addition, there exists a constant $C>0$ such that
	\begin{align}
		\|y\|_{Y_{\frac{1}{4}}} \leq C \left( \|y_0\|_{L^2(0,L)} + \|f \|_{F} \right) .
	\end{align}  
\end{lemma}

\begin{lemma}[{\cite[Section 2.3.1]{Glass-Guerrero}}]\label{Prop-2-well}
For given $y_0\in H^3(0,L)$  with  $y_0(0)=y_0(L)=y_0^\prime(L)=0$, and $f\in F$ with   $F=X_{1}$ or $\widetilde X_{1}$, the system \eqref{single-KdV} admits a unique solution $y\in Y_{1}$. In addition, there exists a constant $C>0$ such that
\begin{align}
	\|y\|_{Y_{1}} \leq C \left( \|y_0\|_{H^3(0,L)} + \|f \|_{F} \right) .
\end{align}  
\end{lemma}

\begin{lemma}[{\cite[Section 2.3.2]{Glass-Guerrero}}]\label{Prop-3-well}
	Let $y_0\equiv 0$ and $\mu\in [1/4,1]$. Then, for given $f\in F$ with   $F=X_{\mu}$ or $\widetilde X_{\mu}$, the system \eqref{single-KdV} admits a unique solution $y\in Y_{\mu}$, and  moreover, there exists a constant $C>0$ such that
	\begin{align}
		\|y\|_{Y_{\mu}} \leq C \|f \|_{F}  .
	\end{align}  
\end{lemma}

Note that the above results are also applicable for the adjoint equation  to \eqref{single-KdV} which is backward in time.

\subsection{Well-posedness of the $4\times 4$ linearized system and its adjoint}

We state the following results. 

\begin{proposition}\label{proposition-linear-well}
	Let $(u_0, v_0)\in [L^2(0,L)]^2$, $(h_1, h_2) \in [L^2((0,T)\times \omega)]^2$ and $(f_1, f_2, f_3,f_4)\in [F]^4$ with $F=X_{\frac{1}{4}}$ or $\widetilde X_{\frac{1}{4}}$ be given. Then, the system \eqref{sys_linear-1}--\eqref{sys_linear-2} possesses a unique solution $(u,v,p,q)\in \big[Y_{\frac{1}{4}}\big]^4$. In addition, there exists a constant $C>0$ such that 
	\begin{equation}\label{estimate-linear}
		\begin{aligned}
			\|(u,v,p,q)\|_{[Y_{\frac{1}{4}}]^4} \leq C \Big( \|(u_0,v_0)\|_{[L^2(0,L)]^2} + \|(h_1,h_2)  \|_{[L^2((0,T)\times \omega)]^2} + \|(f_1,f_2,f_3,f_4)  \|_{[F]^4}      \Big) ,
		\end{aligned}
	\end{equation}     
where $Y_{\frac{1}{4}}$ is defined by \eqref{spaces-1/4}. 
\end{proposition}

\begin{proof} 
The proof of above proposition can be made in the following way. First, we use \Cref{Prop-1-well} to the set of equations \eqref{sys_linear-1} to  show  that $(u,v)\in \big[Y_{\frac{1}{4}} \big]^2$ along with the estimate 
	\begin{equation}\label{esti-u-v}
	\begin{aligned}
		\|(u,v)\|_{[Y_{\frac{1}{4}}]^2} \leq C \Big( \|(u_0,v_0)\|_{[L^2(0,L)]^2} + \|(h_1,h_2)  \|_{[L^2((0,T)\times \omega)]^2} + \|(f_1,f_2)  \|_{[F]^2}      \Big) .
	\end{aligned}
\end{equation} 
Then, using $(u\mathds{1}_\oo, v\mathds{1}_\oo)$ as  source terms in the equations of $(p,q)$ given by \eqref{sys_linear-2}, and combining with \eqref{esti-u-v} we get the required estimate \eqref{estimate-linear}.
\end{proof}

Similar result holds for the adjoint system \eqref{adj-extended-1}--\eqref{adj-extended-2}. 
\begin{proposition}\label{proposition-linear-well-adj}
	Let $(\zeta_0, \theta_0)\in [L^2(0,L)]^2$
	 and $(g_1, g_2, g_3,g_4) \in [F]^4$ with $F=X_{\frac{1}{4}}$ or $\widetilde X_{\frac{1}{4}}$ be given. Then, the system \eqref{adj-extended-1}--\eqref{adj-extended-2} admits a unique solution $(\eta,\psi,\zeta,\theta)\in \big[Y_{\frac{1}{4}}\big]^4$ and moreover,  there exists a constant $C>0$ such that 
	\begin{equation}\label{estimate-linear-adj}
		\begin{aligned}
			\|(\eta,\psi,\zeta,\theta)\|_{[Y_{\frac{1}{4}}]^4} \leq C \Big( \|(\zeta_0,\theta_0)\|_{[L^2(0,L)]^2} +  \|(g_1,g_2,g_3,g_4)  \|_{[F]^4}      \Big) .
		\end{aligned}
	\end{equation}     
\end{proposition}

\medskip 

\subsection{Well-posedness of the $4\times 4$ nonlinear system}

Using a fixed point theorem,  we now prove  the  well-posedness of our $4\times 4$  nonlinear system \eqref{sys_equiv_1}--\eqref{sys_equiv_2}.  

\begin{proposition}\label{Prop-nonlinear-well}
Let $T>0$ and $L>0$.  Then,	there exists some positive real number $\delta_0$ such that for every $(u_0,v_0)\in [L^2(0,L)]^2$, $(h_1, h_2) \in [L^2((0,T)\times \omega)]^2$ and $(\xi_1, \xi_2) \in [L^2(Q_T)]^2$, satisfying 
\begin{align}\label{assumption-1}
\|(u_0,v_0)\|_{[L^2(0,L)]^2} + \|(h_1,h_2)  \|_{[L^2((0,T)\times \omega)]^2} + \|(\xi_1,\xi_2)  \|_{[L^2(Q_T)]^2} \leq \delta_0 , 
\end{align}
  the system \eqref{sys_equiv_1}--\eqref{sys_equiv_2} possesses a unique solution   
  \begin{align*}
  	(u,v,p,q) \in  \big[ Y_{\frac{1}{4}} \big]^4 ,
  \end{align*}
where $Y_{\frac{1}{4}}$ is defined by \eqref{spaces-1/4}. 
\end{proposition}

\smallskip 

Before going to the proof of above proposition,   we  prove the following lemma. 


\begin{lemma}\label{Lemma-nonlinear-conti}
	Let $y_1 , y_2\in L^2(0,T; H^1(0,L))$. Then, $y_1 y_{2,x} \in L^1(0,T; L^2(0,L))$ and the map 
	\begin{align}\label{map-cont}
		(y_1,  y_2) \in [L^2(0,T; H^1(0,L))]^2 \mapsto y_1  y_{2,x} \in L^1(0,T;L^2(0,L) ) 
	\end{align}
is continuous.  
\end{lemma}

\begin{proof}
	Consider any  $(y_1, y_2)$ and $(\widetilde y_1, \widetilde y_2)$ from the space $[L^2(0,T; H^1(0,L))]^2$. Then, we have 
\begin{equation}\label{bound-yy_x}
\begin{aligned} 
\|y_1y_{2,x}\|_{L^1(0,T; L^2(0,L))} 
&= \int_0^T \|y_1 y_{2,x}\|_{L^2(0,L)}  \\
&\leq \int_0^T \|y_1\|_{L^\infty(0,L)} \|y_{2,x}\|_{L^2(0,L)} 
\\
& \leq C_0 \int_0^T \|y_1\|_{H^1(0,L)} \|y_2\|_{H^1(0,L)} \\
& \leq C_0 \|y_1\|_{L^2(0,T; H^1(0,L))} \|y_2\|_{L^2(0,T; H^1(0,L))},  
	\end{aligned}
\end{equation}
for some constant $C_0>0$, 
	which yields the first result of the lemma.
	
\smallskip 

Next, we compute that
	\begin{equation}\label{bound-differebce-yy_x}
		\begin{aligned}
			&\|y_1y_{2,x}  - \y_1 \y_{2,x} \|_{L^1(0,T;L^2(0,L))} \\ 
			&\leq 
			\int_0^T \| y_1 (y_{2,x} - \y_{2,x} )\|_{L^2(0,L)}  
			 + \int_0^T \| (y_{1} - \y_{1} ) \y_{2,x}\|_{L^2(0,L)} \\
		& \leq  \int_0^T \| y_1 \|_{L^\infty(0,L)} \|y_{2,x} - \y_{2,x} \|_{L^2(0,L)} + \int_0^T \| y_{1} - \y_{1} \|_{L^\infty(0,L)} \|\y_{2,x}\|_{L^2(0,L)} \\
		& \leq \int_0^T \| y_1 \|_{H^1(0,L)} \|y_{2} - \y_{2} \|_{H^1(0,L)} + \int_0^T \| y_{1} - \y_{1} \|_{H^1(0,L)} \|\y_{2}\|_{H^1(0,L)} \\
		& \leq C_1 \left(\|y_1\|_{L^2(H^1)} + \| \y_2\|_{L^2(H^1)}    \right) \, \left( \| y_1 - \y_1\|_{L^2(H^1)} + \| y_2 - \y_2\|_{L^2(H^1)}\right)    ,
		\end{aligned} 
	\end{equation}
for some constant $C_1>0$.
This gives the continuity of the map \eqref{map-cont}. 

Therefore, the proof of \Cref{Lemma-nonlinear-conti} is  complete.
\end{proof}

\smallskip

\begin{proof}[\bf Proof of \Cref{Prop-nonlinear-well}]
	
We now prove the well-posedness of our extended system \eqref{sys_equiv_1}--\eqref{sys_equiv_2}. 
Let us define the map 
\begin{align}\label{map-fixed}
	\Lambda : \big[ Y_{\frac{1}{4}} \big]^4 \to \big[ Y_{\frac{1}{4}} \big]^4 , \quad  \Lambda(\widetilde u, \widetilde v,\widetilde p,\widetilde q) = (u,v,p,q),
\end{align}
where $(u,v,p,q)$ is the unique solution to \eqref{sys_linear-1}--\eqref{sys_linear-2} with $(u_0, v_0)\in [L^2(0,L)]^2$, $(h_1,h_2)\in [L^2((0,T)\times \omega)]^2$ and 
\begin{align*} 
&f_1=\xi_1 + 3\widetilde u \widetilde u_x - 6 \widetilde v \widetilde v_x, 
 \quad f_2=\xi_2  - 3\widetilde u \widetilde v_x, \\
&f_3=  3\widetilde p \widetilde u_x - 3 \widetilde q \widetilde v_x  , \qquad \ \ \ f_4 = -6\widetilde p \widetilde v_x .
\end{align*}

Then, by means of \Cref{proposition-linear-well} and the bound \eqref{bound-yy_x},   there exists some constant $C_2>0$ such that we have 
\begin{equation}\label{enrgy-est-non}
	\begin{aligned}
		\|(u,v,p,q)\|_{[Y_{\frac{1}{4}}]^4} \leq C_2 \Big( \|(u_0, v_0)\|_{[L^2(0,L)]^2} + \|(h_1,h_2)\|_{[L^2((0,T)\times \omega)]^2} + \|(\xi_1, \xi_2)\|_{[L^2(Q_T)]^2} \\
		 + \|\widetilde u \|^2_{L^2(H^1)} +  \|\widetilde v \|^2_{L^2(H^1)} +  \|\widetilde u \|_{L^2(H^1)}  \|\widetilde v\|_{L^2(H^1)}    + 
		 \|\widetilde p \|_{L^2(H^1)}  \|\widetilde u\|_{L^2(H^1)} \\ 
		 + \|\widetilde q \|_{L^2(H^1)}  \|\widetilde v\|_{L^2(H^1)} +   \|\widetilde p \|_{L^2(H^1)}  \|\widetilde v \|_{L^2(H^1)}  \Big) . 
	\end{aligned}
\end{equation}

Now, denote the set 
\begin{align}
	\B_R:= \Big\{ (u,v,p,q)\in \big[Y_{\frac{1}{4}} \big]^4 \, : \, \|u\|_{Y_{\frac{1}{4}}} +   \|v\|_{Y_{\frac{1}{4}}}  + \|p\|_{Y_{\frac{1}{4}}}  + \|q\|_{Y_{\frac{1}{4}}}    \leq R       \Big\} .
\end{align}

\smallskip 

-- Then, starting   with $(\widetilde u, \widetilde v,\widetilde p,\widetilde q) \in \B_R$, the estimate \eqref{enrgy-est-non} becomes 
\begin{equation}
	\begin{aligned}
		\|(u,v,p,q)\|_{[Y_{\frac{1}{4}}]^4} \leq C_2 \Big( \|(u_0, v_0)\|_{[L^2(0,L)]^2} + \|(h_1,h_2)\|_{[L^2((0,T)\times \omega)]^2} \\
		 + \|(\xi_1, \xi_2)\|_{[L^2(Q_T)]^2}  +  6R^2\Big).
	\end{aligned}
	\end{equation}
In what follows,  if  $R<\frac{1}{6C_2}$ and 
\begin{align*}
	 \|(u_0, v_0)\|_{[L^2(0,L)]^2} + \|(h_1,h_2)\|_{[L^2((0,T)\times \omega)]^2} 
	+ \|(\xi_1, \xi_2)\|_{[L^2(Q_T)]^2} < \frac{R-6C_2 R^2}{C_2},
	\end{align*}
we have  $\Lambda(\B_R)\subset \B_R$, which concludes  the stability of the map $\Lambda$ given by \eqref{map-fixed}  on the set  $\B_R$.  The quantity $\delta_0$ in \eqref{assumption-1} can  be now chosen as follows: 
$$ \delta_0= \frac{R-6C_2 R^2}{C_2}. $$

\smallskip

--  Let us prove that $\Lambda$ is a contraction map. Consider two elements $(\widetilde u, \widetilde v, \widetilde p, \widetilde q)$ and $(\widehat u, \widehat v, \widehat p, \widehat q)$ from the space $\B_R$. We denote the associated solutions to the linearized model \eqref{sys_linear-1}--\eqref{sys_linear-2}   by $(u_1,v_1,p_1,q_1)$  and  $(u_2,v_2,p_2,q_2)$,  respectively    with 
\begin{align*} 
	&f_1=\xi_1 + 3\widetilde u \widetilde u_x - 6 \widetilde v \widetilde v_x, 
	\quad f_2=\xi_2  - 3\widetilde u \widetilde v_x, \\
	&f_3=  3\widetilde p \widetilde u_x - 3 \widetilde q \widetilde v_x  , \qquad \ \ \  f_4 = -6\widetilde p \widetilde v_x ,
\end{align*}
and  
\begin{align*} 
	&f_1=\xi_1 + 3\widehat u \widehat u_x - 6 \widehat v \widehat v_x, 
	\quad f_2=\xi_2  - 3\widehat u \widehat v_x, \\
	&f_3=  3\widehat p \widehat u_x - 3 \widehat q \widehat v_x  , \qquad \ \ \ f_4 = -6\widehat p \widehat v_x ,
\end{align*}
We further denote $(u,v,p,q)= (u_1-u_2,  v_1-v_2, p_1-p_2, q_1-q_2)$ which satisfies the following set of equations
\begin{align}\label{sys_diff_1}
	&\begin{cases}
		u_t -\frac{1}{2}u_{xxx} = 3(\widetilde u \widetilde u_x -  \widehat u \widehat u_x)   -  6(\widetilde v \widetilde v_x - \widehat v \widehat v_x) 
		 &  \text{in } Q_T, \\
		v_t + v_{xxx} = - 3(\widetilde u\widetilde v_x - \widehat u \widehat v_x)   
		  &  \text{in } Q_T, \\
		u(t,0) = u(t,L) = u_x(t,0) = 0 & \text{for } t\in (0,T), \\
		v(t,0) = v(t,L) = v_x(t,L) =0 & \text{for } t\in (0,T) , \\
		u(0) = 0, \ \ v(0) = 0 & \text{in } (0,L) ,
	\end{cases} \\
\label{sys_diff_2}
	&\begin{cases} 	
		-p_t + \frac{1}{2}p_{xxx}  = u  \mathds{1}_{\mathcal O}  + 3 (\widetilde p \widetilde u_x - \widehat p \widehat u_x) - 3 (\widetilde q \widetilde v_x - \widehat q \widehat v_x)  
		 &  \text{in } Q_T, \\
		-q_t - q_{xxx} =  v \mathds{1}_{\mathcal O} - 6(\widetilde p \widetilde v_x - \widehat p \widehat v_x)  &  \text{in } Q_T, \\
		p(t,0) = p(t,L) = p_x(t,L) = 0 & \text{for } t\in (0,T), \\
		q(t,0) = q(t,L) = q_x(t,0) =0 & \text{for } t\in (0,T) , \\
		p(T)=0, \ \ q(T)=0 & \text{in } (0,L).
	\end{cases}
\end{align}

Thanks to the estimate \eqref{bound-differebce-yy_x},  we have 
\begin{equation*}
	\begin{aligned} 
&	\|\widetilde u  \widetilde u_x - \widehat u \widehat u_x\|_{L^1(L^2)} 
	\leq  2C_1 \left(\|\widetilde u\|_{L^2(H^1)} + \| \widehat u\|_{L^2(H^1)}    \right) \|\widetilde u - \widehat u\|_{L^2(H^1)} ,
\\
& \|\widetilde v  \widetilde v_x - \widehat v \widehat v_x\|_{L^1(L^2)} 
\leq  2C_1 \left(\|\widetilde v\|_{L^2(H^1)} + \| \widehat v\|_{L^2(H^1)}    \right) \|\widetilde v - \widehat v\|_{L^2(H^1)} ,\\
& \|\widetilde u  \widetilde v_x - \widehat u \widehat v_x\|_{L^1(L^2)} 
\leq  C_1 \left(\|\widetilde u\|_{L^2(H^1)} + \| \widehat v\|_{L^2(H^1)}    \right) \left(\|\widetilde u - \widehat u\|_{L^2(H^1)} + \|\widetilde v - \widehat v\|_{L^2(H^1)} \right)  ,\\
& \|\widetilde p  \widetilde u_x - \widehat p \widehat u_x\|_{L^1(L^2)} 
\leq  C_1 \left(\|\widetilde p\|_{L^2(H^1)} + \| \widehat u\|_{L^2(H^1)}    \right) \left(\|\widetilde p - \widehat p\|_{L^2(H^1)} + \|\widetilde u - \widehat u\|_{L^2(H^1)} \right), \\
&  \|\widetilde q  \widetilde v_x - \widehat q \widehat v_x\|_{L^1(L^2)} 
\leq  C_1 \left(\|\widetilde q \|_{L^2(H^1)} + \| \widehat v\|_{L^2(H^1)}    \right) \left(\|\widetilde q - \widehat q\|_{L^2(H^1)} + \|\widetilde v - \widehat v\|_{L^2(H^1)} \right) , \\
&  \|\widetilde p  \widetilde v_x - \widehat p \widehat v_x\|_{L^1(L^2)} 
\leq  C_1 \left(\|\widetilde p \|_{L^2(H^1)} + \| \widehat v\|_{L^2(H^1)}    \right) \left(\|\widetilde p - \widehat p\|_{L^2(H^1)} + \|\widetilde v - \widehat v\|_{L^2(H^1)} \right) ,
\end{aligned}
\end{equation*}
where the constant $C_1>0$ is the same as appeared in \eqref{bound-differebce-yy_x}. 

Using the above information and by \Cref{proposition-linear-well}, we can say that there exists some constant $C_3>0$ such that  the solution to \eqref{sys_diff_1}--\eqref{sys_diff_2} satisfies 
\begin{equation}
	\begin{aligned}
	&\| (u,v,p,q)\|_{[Y_{\frac{1}{4}}]^4} \\
	& \leq C_3  \left( \|(\widetilde u, \widetilde v, \widetilde p, \widetilde q)\|_{[L^2(H^1)]^4} + \|(\widehat u, \widehat v, \widehat p, \widehat q)\|_{[L^2(H^1)]^4} \right)  \|(\widetilde u, \widetilde v, \widetilde p, \widetilde q) - (\widehat u, \widehat v, \widehat p, \widehat q)\|_{[L^2(H^1)]^4}  \\
	&\leq  2C_3R \,  \|(\widetilde u, \widetilde v, \widetilde p, \widetilde q) - (\widehat u, \widehat v, \widehat p, \widehat q)\|_{[L^2(H^1)]^4}
	\end{aligned}
\end{equation}
Now,  choose $R>0$ in such a way that $2C_3 R<1$, so that  the map $\Lambda$ is contracting. Therefore, using the Banach fixed point theorem,  there exists a unique fixed point  of $\Lambda$ in $\B_R$, which is actually the  unique solution $(u,v,p,q)$ to \eqref{sys_equiv_1}--\eqref{sys_equiv_2}. 

The proof is complete.
\end{proof}


\section{Carleman estimates}\label{Sec-Carleman}
This  section is devoted to obtain a suitable  Carleman estimate satisfied by the solution  to our  adjoint system \eqref{adj-extended-1}--\eqref{adj-extended-2}.


\subsection{Choice of Carleman weights}\label{Section-weights}
   
 Recall that $\oo \cap \omega\neq \emptyset$ and $\omega_0\subset \subset \oo \cap \omega$. Assume that $\omega_0=(l_0,l_1)$ with $0<l_0<l_1<L$, and set $l_{1/2}={(l_0+l_1)}/{2}$. We now consider the weight functions as introduced in \cite[Section 3]{Cerpa-Kdv_Schrodinger} (see also \cite{Carreno-Hirota-Satsuma}): for $K_1, K_2>0$ (to be specified later), define the smooth functions 
\begin{align}\label{Weight-1}
	\beta(x) = 1+K_1 \left(1-e^{ -K_2(x- l_{1/2} )^2}  \right) , \quad \ \xi(t) = \frac{1}{t(T-t)} , \qquad \forall x\in [0,L], \ \ \ \forall t\in (0,T),
\end{align}
and 
\begin{align}\label{Weight-2}
	\vphi(t,x) = \xi(t)\beta(x) , \qquad \forall (t,x) \in (0,T) \times [0,L].  
\end{align}

For any $K_1,K_2>0$, we note that $\beta >0$ in $[0,L]$ and consequently,  $\vphi>0$ in $(0,T) \times [0,L]$.  We further observe that
\begin{equation}\label{property-beta-x}
\begin{aligned}
\text{there exists some } & \ c>0 \ \text{ such that } \	|\beta_x| \geq  c > 0   \quad  \text{in } \ [0,L] \setminus \overline{\omega_0} , \\
\text{and }\  & \beta_x(0) <0 , \quad \beta_x(L) >0    .
\end{aligned}
\end{equation}
Also, one can choose $K_1$ and $K_2$ in such a way that 
\begin{align}\label{property-beta-xx}
	\beta_{xx} <0 \quad \ \text{in } \ [0,L] \setminus \overline{\omega_0} .
\end{align}
Indeed, the property \eqref{property-beta-xx} holds true if we set
\begin{align}\label{choice-K-2}
 K_2 =\frac{4}{(l_1-l_0)^2}. 
 \end{align}

We further   consider 
\begin{equation}\label{Weight-min-max} 
\begin{aligned}
	&\vphi^*(t) : = \min_{[0,L]} \vphi(t,x)  = \xi(t)\beta(l_{1/2}) = \xi(t), \quad \forall t \in (0,T), \\
	& \widehat \vphi(t) := \max_{[0,L]} \vphi(t,x)  = \xi(t) \big(\max\left\{ \beta(0), \beta(L) \right\}\big), \quad \forall t \in (0,T) .
\end{aligned}
\end{equation}
  Now, denote 
\begin{align*}
	M(K_2, l_{1/2}) : = \max \big\{ 1- e^{-K_2 l^2_{1/2}}, \ 1- e^{-K_2 (L-l_{1/2})^2 }    \big\} .
\end{align*}
 Then, there  exists  some constant $c_0>0$ such that   the  weight functions in \eqref{Weight-min-max}  verify the following criterion:
\begin{align}\label{cond-weights-max-min}
	36 s \vphi^*(t) - 35 s\widehat \vphi(t) \geq c_0 s \xi(t) , \quad \forall t \in (0,T), 
\end{align}
provided we choose $0 < K_1 < \frac{1}{35 M(K_2, l_{1/2})}$; in particular, we set
\begin{align}\label{choice-K-1}
	K_1 = \frac{1}{70 M(K_2, l_{1/2})} .
\end{align} 


%

\subsection{Carleman estimates for the single KdV equation}
Let us prescribe a Carleman estimate for the following system 
\begin{align}\label{adj-gen}  
	\begin{cases}
		z_t \pm z_{xxx} = g  & \text{in } Q_T,  \\
		z(t,0) = z(t,L) = z_x(t,0) = 0  & \text{for } t\in (0,T) ,\\
		z(T) = z_T &\text{in } (0,L)  .
	\end{cases}
\end{align}
with given right hand side $g$ and final data $z_T$. One can also consider the boundary conditions: 
	\begin{align}\label{bdry-gen-2}
	 z(t,0) = z(t,L) = z_x(t,L) = 0 \quad \text{for }t\in (0,T) ,
	\end{align}
as a replacement for  the set of boundary conditions  in \eqref{adj-gen}, and  this will  not affect the underlying Carleman estimate. 

We hereby recall a Carleman estimate for the linear KdV equation which has been obtained for instance in \cite[Theorem 3.1]{Cerpa-Kdv_Schrodinger}; see also  \cite[Proposition 3.1]{Roberto-Rosier-internal-kdv}.  
\begin{proposition}\label{prop-carleman-kdv}
	Let $T>0$ be given  and $\omega_0\subset (0,L)$ be a non-empty open set as introduced in Section \ref{Section-weights}. Then, there exist constants $C>0$ and $s_0>0$ such that for any $g\in L^2(Q_T)$ and $z_T \in L^2(0,L)$, the solution $z$ to    \eqref{adj-gen}
	 satisfies 
	\begin{multline}\label{carlemn-gen}
		s^5 \iint_{Q_T} e^{-2s\vphi}\xi^5 |z|^2 + 	s^3 \iint_{Q_T} e^{-2s\vphi}\xi^3 |z_x|^2 + 	s \iint_{Q_T} e^{-2s\vphi}\xi \, |z_{xx}|^2 \\
		\leq C \bigg(\iint_{Q_T} e^{-2s\vphi} |g|^2  + \int_0^T \int_{\omega_0}  e^{-2s\vphi} \Big[ s^5\xi^5  |z|^2  + s\xi \, |z_{xx}|^2 \Big]   \bigg)     ,
	\end{multline}
	for all $s\geq s_0$. 
\end{proposition}

Now by using the above proposition,   we can obtain a modified Carleman  inequality  for \eqref{adj-gen} with more regular right hand side $g$; see \eqref{carleman-desired} below.  Although, similar  result   has already been addressed for instance in \cite{Roberto-Rosier-internal-kdv,Carreno-Hirota-Satsuma},  we give a sketch of the proof for  sake of completeness. More precisely, we prove the following proposition.

	%
\begin{proposition}\label{prop-carleman-regular}
	Let $T>0$ be given  and $\omega_0\subset (0,L)$ be a non-empty open set as introduced in Section \ref{Section-weights}.  Also, assume that $\nu\in (0,1]$. Then, there exist constants $C>0$ and $s_0>0$ such that for any $g\in L^2(0,T; H^{\nu}(0,L))$  and $z_T \in L^2(0,L)$, the solution $z$ to    \eqref{adj-gen} satisfies 
\begin{equation}\label{carleman-desired}
	\begin{aligned} 
	& s^5 \iint_{Q_T} e^{-2s\vphi}\xi^5 |z|^2 + 	s^3 \iint_{Q_T} e^{-2s\vphi}\xi^3 |z_x|^2  \\
	& \qquad  \qquad \qquad + 	s \iint_{Q_T} e^{-2s\vphi}\xi |z_{xx}|^2 
	 +  s\int_0^T e^{-2s\widehat \vphi} \xi^{-3} \|z\|^2_{H^{2+\nu}(0,L)}   \\
	 &\leq 
	  C s^3 \iint_{Q_T} e^{-2s\vphi} \xi |g|^2  + C s\int_0^T e^{-2s\widehat \vphi} \xi^{-3} \|g\|^2_{H^{\nu}(0,L)}  \\
& \qquad \qquad 	+ C s^5 \int_0^T \int_{\omega_0} e^{-2s\vphi} \xi^{5}  |z|^2 + C s \int_0^T   \int_{\omega_0}  e^{-2s (1+ \frac{2}{\nu}) \vphi^* + \frac{4s}{\nu}\widehat \vphi  }   \,  \xi^{1+ \frac{8}{\nu}}  |z|^2 ,  
	\end{aligned}     
\end{equation}
for all $s\geq s_0$. 
\end{proposition}

\begin{proof}
	Let us recall the definition of $\vphi^*$ from \eqref{Weight-min-max}, so that we can rewrite the Carleman estimate \eqref{carlemn-gen} as follows:
		\begin{multline}\label{carlemn-gen-1}
		s^5 \iint_{Q_T} e^{-2s\vphi}\xi^5 |z|^2 + 	s^3 \iint_{Q_T} e^{-2s\vphi}\xi^3 |z_x|^2 + 	s \iint_{Q_T} e^{-2s\vphi}\xi |z_{xx}|^2 \\
		\leq C \bigg(\iint_{Q_T} e^{-2s\vphi} |g|^2 + s^5 \int_0^T \int_{\omega_0}  e^{-2s\vphi}  \xi^5  |z|^2  + s \int_0^T \int_{\omega_0}  e^{-2s\vphi^*}  \xi \, |z_{xx}|^2    \bigg)     ,
	\end{multline}
	for all $s\geq s_0$. 
	
	Our aim is to eliminate the observation integral of  $z_{xx}$ by applying the so-called bootstrap argument. Denote 
	\begin{align}\label{term-J}
	J := s \int_0^T \int_{\omega_0} e^{-2s \vphi^*} \xi \, |z_{xx}|^2 . 
	\end{align}
	Thanks to the fact that  $\xi$ and $\vphi^*$ do not depend on space variable,	 we have 
	\begin{align}\label{term-J-bound}
		J \leq s \int_0^T  e^{-2s \vphi^*} \xi \, \|z\|_{H^2(\omega_0)}^2 .
	\end{align}

	Let $\nu \in (0,1]$. Writing $H^2(\omega_0)$ as  an interpolation  between the spaces $H^{2+\nu}(\omega_0)$ and $L^2(\omega_0)$, and then applying the Young's inequality, we eventually get 
	\begin{align}
	J &\leq s\int_0^T e^{-2s \vphi^*} \xi \, \|z \|^{\frac{4}{(2+\nu)}}_{H^{2+\nu}(\omega_0)}  \| z \|^{\frac{2\nu}{(2+\nu)}}_{L^2(\omega_0)}    
	\notag \\
& \leq \epsilon s \int_0^T e^{-2s\widehat \vphi} \xi^{-3} \|z\|^2_{H^{2+\nu}(\omega_0)} + C_\epsilon \, s \int_0^T
   e^{-2s (1+ \frac{2}{\nu}) \vphi^* + \frac{4s}{\nu}\widehat \vphi  }   \,  \xi^{1+ \frac{8}{\nu}}      \|z\|^2_{L^2(\omega_0)}  , \label{aux-1}
	\end{align}
for any chosen $\epsilon>0$.

 Now we need to determine a proper estimate for the first term in the r.h.s. of \eqref{aux-1}. 	Following the same argument as developed in \cite{Glass-Guerrero} (see also   \cite{Cerpa-Kdv_Schrodinger,Roberto-Rosier-internal-kdv}), let us employ the bootstrap technique. We define $\widehat z(t,x):= \rho_1(t) z(t,x)$ with $\rho_1(t)= s^{1/2}\xi^{1/2} e^{-s\widehat \vphi}$ so that $\widehat z$ satisfies 
\begin{align}\label{equation-hat-z}
	\begin{cases}
		\widehat z_t \pm \widehat z_{xxx} =  g_1:= \rho_1  g + \rho_{1,t} z & \text{in } Q_T, \\
		 	\widehat z(t,0) = 	\widehat z(t,L) = 	\widehat z_x(t,0) = 0 & \text{for } t\in (0,T), \\
		 	\widehat z(T) = 0	 .
	\end{cases}
\end{align}
 Note that 
\begin{align*}
	|\rho_{1,t}| \leq C s^{3/2} \xi^{5/2} e^{-s\widehat \vphi}.
	\end{align*}
Using this, we deduce  
\begin{align}\label{bound-g-1} 
	\|g_1 \|^2_{L^2(Q_T)} \leq  C s\iint_{Q_T} e^{-2s\widehat \vphi}  \xi |g|^2 + C s^3 \int_{Q_T} e^{-2s\widehat \vphi} \xi^5 |z|^2 ,
\end{align}
	and therefore, by virtue of \Cref{Prop-3-well}, we have  $\widehat z \in Y_{\frac{1}{2}}$, where the space $Y_{\frac{1}{2}}$ is introduced in \eqref{spaces-1/2}.
In particular, one has 
\begin{align}\label{bound-hat-z}
	\|\widehat z \|^2_{L^2(0,T; H^2(0,L))} \leq 
	C 	\|g_1 \|^2_{L^2(Q_T)}  .
\end{align}
	
\medskip 
	
Next, we define $\widetilde z(t,x)= \rho_2(t) z(t,x)$ with $\rho_2(t)= s^{1/2}\xi^{-3/2}e^{-s\widehat \vphi}$. Then $\widetilde z$ satisfies 
\begin{align}\label{equation-tilde-z}
	\begin{cases}
		\widetilde z_t \pm \widetilde z_{xxx} =  g_2: = \rho_2  g + \rho_{2,t} \rho^{-1}_{1} \widehat z   & \text{in } Q_T, \\
		\widetilde z(t,0) = 	\widetilde z(t,L) = 	\widetilde z_x(t,0) = 0 & \text{for } t\in (0,T), \\
		\widetilde z(T) = 0	 .
	\end{cases}
\end{align}
Let us compute that 
 $$|\rho_{2,t}\rho^{-1}_1| \leq Cs.$$
 Then, thanks to the fact that  $\widehat z\in Y_{\frac{1}{2}}$ and $g\in L^2(0,T; H^\nu(0,L))$, one has $g_2 \in L^2(0,T; H^\nu(0,L))$ ($=X_{\frac{1}{2} +\frac{\nu}{4}}$ as defined by \eqref{spaces-nu/4}).  
 As a consequence,  by \Cref{Prop-3-well} we have
$$  \widetilde z  \in Y_{\frac{1}{2} +\frac{\nu}{4}} = L^2(0,T; H^{2+\nu}(0,L)) \cap \C^0([0,T]; H^{1+\nu}(0,L)),$$
 with its estimate
 \begin{equation} 
 \begin{aligned}\label{bound-tilde-z}
 	\| \widetilde z \|^2_{L^2(H^{2+\nu})  \cap L^\infty(H^{1+\nu}) } &\leq C \| g_2 \|^2_{L^2(H^\nu)} \\
 	& \leq C s \int_0^T  e^{-2s\widehat \vphi} \xi^{-3} \|g\|^2_{H^{\nu}(0,L)} + C s^2 \| \widehat z\|^2_{L^2(0,T; H^2(0,L))} 	  .
  \end{aligned}
\end{equation}
 Thereafter, from \eqref{bound-g-1},   \eqref{bound-hat-z} and \eqref{bound-tilde-z} we obtain
 \begin{equation}
  \begin{aligned}
  	&s\int_0^T e^{-2s\widehat \vphi} \xi^{-3} \|z\|^2_{H^{2+\nu}(0,L)} \\
  	& \leq C s \int_0^T  e^{-2s\widehat \vphi} \xi^{-3} \|g\|^2_{H^{\nu}(0,L)} + C s^3 \iint_{Q_T} \xi e^{-2s\widehat \vphi} |g|^2 + C s^5 \iint_{Q_T} e^{-2s\widehat \vphi} \xi^5 |z|^2       .
  	\end{aligned} 
	\end{equation}
Utilizing the above estimate in \eqref{aux-1} and combining  with  \eqref{term-J-bound}, \eqref{term-J} and \eqref{carlemn-gen-1} we get the desired Carleman estimate \eqref{carleman-desired} by choosing $\epsilon>0$ small enough. 

The proof is finished. 
\end{proof}	
	
\smallskip 
	
	\subsection{Carleman estimate for the $4\times 4$ adjoint system}
	
	We are now in position to prove a Carleman estimate for the $4\times 4$ adjoint system \eqref{adj-extended-1}--\eqref{adj-extended-2}.  Consider  $\nu=1$ in \Cref{prop-carleman-regular} and denote 
	\begin{equation}\label{notation-I} 
	\begin{aligned}
		 I(z, s) 
	 : = 	s^5 \iint_{Q_T} e^{-2s\vphi}\xi^5 |z|^2 + 	s^3 \iint_{Q_T} e^{-2s\vphi}\xi^3 |z_x|^2 + 	s \iint_{Q_T} e^{-2s\vphi}\xi \, |z_{xx}|^2 \\ + s\int_0^T e^{-2s\widehat \vphi} \xi^{-3} \|z\|^2_{H^{3}(0,L)} ,
	\end{aligned}
	\end{equation}
	so that the inequality \eqref{carleman-desired} becomes
	\begin{equation}\label{carleman-refined}
	\begin{aligned}
		I(z,s) 
		 &\leq C s^3 \iint_{Q_T} e^{-2s\vphi} \xi |g|^2  + C s\int_0^T e^{-2s\widehat \vphi} \xi^{-3} \|g\|^2_{H^{1}(0,L)}  \\
		& \qquad \qquad 	+ C s^5 \int_0^T \int_{\omega_0} e^{-2s\vphi} \xi^{5}  |z|^2 + C s \int_0^T   \int_{\omega_0}  e^{-6s  \vphi^* + 4s\widehat \vphi  }   \,  \xi^{9}  |z|^2 \\
& \leq C s^3 \iint_{Q_T} e^{-2s\vphi} \xi |g|^2  + C s\int_0^T e^{-2s\widehat \vphi} \xi^{-3} \|g\|^2_{H^{1}(0,L)}  
+C s^5 \int_0^T   \int_{\omega_0}  e^{-6s  \vphi^* + 4s\widehat \vphi  }   \,  \xi^{9}  |z|^2 ,
	\end{aligned}
	\end{equation}
	for all $s\geq s_0$.

	\smallskip 
	
Let us  state and prove the main result of the current section. From now onwards,  we shall consider the source terms $g_i \in L^2(0,T; H^1_0(0,L))$ for each $i\in \{1,2,3,4\}$ in the adjoint system \eqref{adj-extended-1}--\eqref{adj-extended-2}.
\begin{theorem}[Main Carleman estimate]\label{Carleman-main}
		Let $T>0$ be given  and $\omega_0\subset (0,L)$ be a non-empty open set as introduced in Section \ref{Section-weights}. Then there exist constants $C>0$ and $\widehat s_0:=\widehat s_0(T)>0$ such that for any given   source terms $g_i\in L^2(0,T; H^1_0(0,L))$ for $i=1,2,3,4$, and $(\zeta_0,\theta_0) \in [L^2(0,L)]^2$, the solution to \eqref{adj-extended-1}--\eqref{adj-extended-2} satisfies     
	\begin{equation}\label{Carleman-main-adjoint}
		\begin{aligned}
			&	I(\eta,s) + I(\psi, s) + I(\zeta, s) + I(\theta, s) \\
			& \leq 
			 C s^5 \iint_{Q_T} e^{-12s\vphi^*+ 10s\widehat \vphi} \xi^{13} \big(|g_{1,x}|^2 + |g_{2,x}|^2 + |g_{3,x}|^2 + |g_{4,x}|^2  \big) \\
			& \ \  + C s^{25} \int_0^T \int_{\omega_0} e^{-36s\vphi^* + 34s\widehat \vphi} \xi^{57} |\eta|^2 + C s^{25} \int_0^T \int_{\omega_0} e^{-36s\vphi^* + 34s\widehat \vphi} \xi^{57} |\psi|^2 ,
		\end{aligned}
	\end{equation}
for all $s\geq \widehat s_0$, where $I(\cdot, s)$ has been defined in \eqref{notation-I}. 
\end{theorem}	
	

\begin{proof}
	We apply \Cref{prop-carleman-regular}, more specifically the estimate \eqref{carleman-refined} to each of the states $\eta$, $\psi$, $\zeta$ and $\theta$ of the adjoint system \eqref{adj-extended-1}--\eqref{adj-extended-2}. 
	
	In what follows, $\eta$ satisfies
  \begin{equation}\label{est-eta}
  	\begin{aligned}
  		I(\eta, s) & \leq 
  		C s^3 \iint_{Q_T} e^{-2s\vphi} \xi \big(|\zeta|^2 + |g_1|^2\big)   
  		+ C s\int_0^T e^{-2s\widehat \vphi} \xi^{-3} \big( \|\zeta\|^2_{H^1(0,L)} + \|g_1\|^2_{H^{1}(0,L)}  \big) \\
  		& \qquad +C s^5 \int_0^T   \int_{\omega_0}  e^{-6s  \vphi^* + 4s\widehat \vphi  }   \,  \xi^{9}  |\eta|^2  \\
  		& 
  		\leq 
  	C s^3 \iint_{Q_T} e^{-2s\vphi} \xi \big(|\zeta|^2 + |g_1|^2\big)   
  		+ C s\iint_{Q_T} e^{-2s\widehat \vphi} \xi^{-3} \big( |\zeta_x|^2 + |g_{1,x}|^2  \big) \\
  	& \qquad + C s^5 \int_0^T   \int_{\omega_0}  e^{-6s  \vphi^* + 4s\widehat \vphi  }   \,  \xi^{9}  |\eta|^2 ,
  	\end{aligned}
  \end{equation}
for all $s\geq s_0$, where we have used the Poincar\'{e} inequality since we  have  $\zeta \in L^2(0,T;H^1_0(0,L))$ and  the source term $g_1$ belongs to  $L^2(0,T; H^1_0(0,L))$. 
	
\medskip

Similarly, $\psi$, $\zeta$ and $\theta$ satisfy
\begin{align}
			\label{est-psi}
	&	I(\psi, s) 
		 \leq 	C s^3 \iint_{Q_T} e^{-2s\vphi} \xi \big(|\theta|^2 + |g_2|^2\big)   
		+ C s\iint_{Q_T} e^{-2s\widehat \vphi} \xi^{-3} \big( |\theta_x|^2 + |g_{2,x}|^2  \big)   \notag  \\
		& \qquad \qquad \qquad \qquad + C s^5 \int_0^T   \int_{\omega_0}  e^{-6s  \vphi^* + 4s\widehat \vphi  }   \,  \xi^{9}  |\psi|^2 , \\
		\label{est-zeta}
		&	I(\zeta, s) 
			 \leq 	C s^3 \iint_{Q_T} e^{-2s\vphi} \xi |g_3|^2  
			+ C s\iint_{Q_T} e^{-2s\widehat \vphi} \xi^{-3}
			|g_{3,x}|^2  
			+ C s^5 \int_0^T   \int_{\omega_0}  e^{-6s  \vphi^* + 4s\widehat \vphi  }   \,  \xi^{9}  |\zeta|^2 ,  
			 \\
		\label{est-theta}
	&	I(\theta, s) 
	\leq 	C s^3 \iint_{Q_T} e^{-2s\vphi} \xi |g_4|^2 
	+ C s\iint_{Q_T} e^{-2s\widehat \vphi} \xi^{-3}
	|g_{4,x}|^2  
	+ C s^5 \int_0^T   \int_{\omega_0}  e^{-6s  \vphi^* + 4s\widehat \vphi  }   \,  \xi^{9}  |\theta|^2 		,
	\end{align}
	for all $s\geq s_0$.

	\medskip 
	
{\bf I.} {\underline{\em Absorbing the lower order terms.}}
	Thanks to the fact $1\leq 2 T^2 \xi$ and by adding \eqref{est-eta}, \eqref{est-psi}, \eqref{est-zeta} and \eqref{est-theta}, we obtain 
	
	\begin{equation}\label{carleman-aux-1} 
		\begin{aligned}
		&	I(\eta,s) + I(\psi, s) + I(\zeta, s) + I(\theta, s) \\
		&\leq  C s^3 \iint_{Q_T} e^{-2s\vphi} \xi \big(|g_1|^2 + |g_2|^2 + |g_3|^2 + |g_4|^2  \big)  \\
		&  \ + C s \iint_{Q_T} e^{-2s\widehat \vphi} \xi^{-3}\big(|g_{1,x}|^2 + |g_{2,x}|^2 + |g_{3,x}|^2 + |g_{4,x}|^2  \big) + C s^3 T^8 \iint_{Q_T} e^{-2s\vphi} \xi^5 \big(|\zeta|^2 + |\theta|^2  \big)  \\
		& \ +  C s T^{12} \iint_{Q_T} e^{-2s\vphi} \xi^{3} \big(|\zeta_x|^2 + |\theta_x|^2 \big) 
	+ C s^5 \int_0^T   \int_{\omega_0}  e^{-6s  \vphi^* + 4s\widehat \vphi  }   \,  \xi^{9}  \big(|\eta|^2+ |\psi|^2 + |\zeta|^2 + |\theta|^2 \big) .
		\end{aligned}
	\end{equation} 
Then there is some $c_0>0$ such that 	by taking $s\geq \widehat s_0:= c_0(T^4+T^6)$, we can absorb the integrals $C s^3 T^8 \iint_{Q_T} e^{-2s\vphi} \xi^5 \big(|\zeta|^2 + |\theta|^2  \big)$ and $C s T^{12} \iint_{Q_T} e^{-2s\vphi} \xi^5 \big(|\zeta_x|^2 + |\theta_x|^2  \big)$ by the corresponding leading terms in the left hand side of \eqref{carleman-aux-1}. 
	
	\smallskip 
	

	Recall that $g_i\in L^2(0,T; H^1_0(0,L))$ for $i\in \{1,2,3,4\}$,  and thus by employing the Poincar\'{e} inequality, we get
	\begin{align*}
		\iint_{Q_T} e^{-2s\vphi} \xi |g_i|^2 \leq  	\iint_{Q_T} e^{-2s\vphi^*} \xi |g_i|^2 \leq  C	\iint_{Q_T} e^{-2s\vphi^*} \xi |g_{i,x}|^2 .
	\end{align*}

		\smallskip 
	
	As a consequence of the above information, the inequality \eqref{carleman-aux-1}  now reduces to

	\begin{equation}\label{carleman-aux-2}
	\begin{aligned}
		&	I(\eta,s) + I(\psi, s) + I(\zeta, s) + I(\theta, s) \\
		&\leq  C s^3 \iint_{Q_T} e^{-2s\vphi^*} \xi \big(|g_{1,x}|^2 + |g_{2,x}|^2 + |g_{3,x}|^2 + |g_{4,x}|^2  \big) \\
		& \ \ + C s^5 \int_0^T   \int_{\omega_0}  e^{-6s  \vphi^* + 4s\widehat \vphi  }    \xi^{9}  \big(|\eta|^2+ |\psi|^2 + |\zeta|^2 + |\theta|^2 \big) ,
		\end{aligned}
	\end{equation}
	for all $s\geq \widehat s_0$. 
	
	\medskip 
	{\bf II.} {\underline{\em Absorbing the unusual observation terms.}} In this part, we eliminate the observation integrals associated to $\zeta$ and $\theta$.
	
	\medskip 
	
	$\bullet$ Consider a nonempty open set  $\widehat \omega \subset \subset \omega_0$ and a function
	$\phi \in\C^\infty_c(\omega_0)$ such that $0\leq \phi \leq 1$ in $\omega_0$ and $\phi=1$ in $\widehat \omega$. In fact, one can obtain the estimate \eqref{carleman-aux-2} with the observation domain $\widehat \omega$. From the equation   of $\eta$, that is \eqref{adj-extended-1}$_1$,  we have 
	\begin{align*}
		\zeta = -\eta_t + \frac{1}{2} \eta_{xxx} - g_1 \quad \ \text{in  $\mathcal O$ }  \text{ (consequently in $\omega_0$)} ,
	\end{align*}
	which yields 
	\begin{equation}\label{observation-zeta}
		\begin{aligned}
	s^5 \int_0^T \int_{\widehat \omega}  e^{-6s  \vphi^* + 4s\widehat \vphi  }    \xi^{9} |\zeta|^2  
	&\leq s^5 \int_0^T \int_{\omega_0}  \phi  e^{-6s  \vphi^* + 4s\widehat \vphi  }   \,  \xi^{9} \zeta \left(-\eta_t + \frac{1}{2}\eta_{xxx} - g_1   \right) \\
	& := J_1 + J_2 + J_3 .
		\end{aligned}
	\end{equation}
	
We now  estimate the terms $J_1, J_2$ and  $J_3$. 

\medskip 

(i) {\underline{\em Estimate of $J_1$}}.
Integrating by parts with respect to time $t$, the term $J_1$ becomes 
\begin{equation}\label{term-J-1}
	\begin{aligned}
		J_1 &= s^5 \int_0^T \int_{\omega_0} \phi \big(e^{-6s  \vphi^* + 4s\widehat \vphi  }   \,  \xi^{9} \big)_t  \zeta \eta + s^5 \int_0^T \int_{\omega_0} \phi e^{-6s  \vphi^* + 4s\widehat \vphi  }     \xi^{9}   \zeta_t \eta \\
		& := J_{1,1} + J_{1,2}. 
	\end{aligned}
\end{equation}
Using the fact 
\begin{align}
 \left| \big(e^{-6s  \vphi^* + 4s\widehat \vphi  }     \xi^{9} \big)_t  \right| \leq C T s e^{-6s  \vphi^* + 4s\widehat \vphi  }    \xi^{11}  ,
\end{align}
and the Young's inequality, we first deduce that 
\begin{equation}\label{esti-J-11}
	\begin{aligned}
		|J_{1,1}| & \leq C T s^6 \int_0^T \int_{\omega_0} e^{-6s  \vphi^* + 4s\widehat \vphi  }  \xi^{11} |\zeta \eta | \\
		& \leq \epsilon s^5 \iint_{Q_T} e^{-2s\vphi} \xi^5 |\zeta|^2 + C_\epsilon s^7 \int_0^T \int_{\omega_0} e^{-12s\vphi^* + 10 s\widehat \vphi} \xi^{17} |\eta|^2 
		\end{aligned}
\end{equation}
for given $\epsilon>0$.  

\smallskip 
We now recall the equation of $\zeta$ from \eqref{adj-extended-2}$_1$; by performing an  integrating by parts in space, we have   
\begin{equation}\label{term-J-12}
	\begin{aligned}
		J_{1,2} &=
		    s^5 \int_0^T \int_{\omega_0} \phi e^{-6s  \vphi^* + 4s\widehat \vphi  }  \xi^{9}  \eta \left(\frac{1}{2}\zeta_{xxx} + g_3  \right)   \\
		&= - \frac{1}{2}s^5 \int_0^T \int_{\omega_0} \phi  e^{-6s  \vphi^* + 4s\widehat \vphi  }  \xi^{9}  \eta_x \zeta_{xx} 
		- \frac{1}{2}s^5 \int_0^T \int_{\omega_0} \phi_x  e^{-6s  \vphi^* + 4s\widehat \vphi  }  \xi^{9}  \eta \zeta_{xx}  \\
& \qquad 	+	s^5 \int_0^T \int_{\omega_0} \phi  e^{-6s  \vphi^* + 4s\widehat \vphi  }  \xi^{9}  \eta g_3 \\
	& : = J^{1}_{1,2} + J^2_{1,2} + J^3_{1,2} .	
	\end{aligned}
\end{equation}
Using the Young's inequality, we  compute
\begin{align}\label{esti-J-12-1}
	&|J^1_{1,2}| \leq \epsilon s \iint_{Q_T}  e^{-2s\vphi} \xi |\zeta_{xx}|^2 + C_\epsilon s^9 \int_0^T \int_{\omega_0} e^{-12s\vphi^* + 10s\widehat \vphi} \xi^{17} |\eta_{x}|^2 
\\ 
\label{esti-J-12-2}
&|	J^2_{1,2} | 
\leq \epsilon s \iint_{Q_T} e^{-2s\vphi} \xi |\zeta_{xx}|^2 + C_\epsilon  s^9  \int_0^T \int_{\omega_0} e^{-12 s \vphi^* + 10 s\widehat \vphi} \xi^{17} |\eta|^2 ,       
\end{align} 
for any given $\epsilon>0$. 

We also have 
\begin{align}\label{esti-J-12-3}
		|	J^3_{1,2} | 
		\leq  \epsilon s^5 \iint_{Q_T} e^{-2s\vphi} \xi^5 |\eta|^2 +   C_\epsilon s^5  \iint_{Q_T} e^{-12 s \vphi^* + 10 s\widehat \vphi} \xi^{13} |g_3|^2 ,       
\end{align} 
for any $\epsilon>0$.

Combining \eqref{esti-J-11}, \eqref{esti-J-12-1}, \eqref{esti-J-12-2}, \eqref{esti-J-12-3}   in \eqref{term-J-1}, we obtain
\begin{equation}\label{estimate-J-1}
	\begin{aligned}
		|J_1| & \leq \epsilon s^5 \iint_{Q_T} e^{-2s\vphi} \xi^5 |\zeta|^2 + 2\epsilon s \iint_{Q_T} e^{-2s\vphi} \xi |\zeta_{xx}|^2  +   \epsilon s^5 \iint_{Q_T} e^{-2s\vphi} \xi^5 |\eta|^2  \\
		& \ \ + C_\epsilon s^9 \int_0^T \int_{\omega_0} e^{-12s\vphi^* + 10 s\widehat \vphi} \xi^{17} \left( |\eta|^2 +|\eta_x|^2 \right) + C s^5 \iint_{Q_T} e^{-12 s \vphi^* + 10 s\widehat \vphi} \xi^{13} |g_3|^2 .
	\end{aligned}
\end{equation}

\medskip 

(ii) {\underline{\em Estimate of $J_2$}}. We recall the term $J_2$ from \eqref{observation-zeta}. Integration by parts in space, we get 
\begin{equation}
	\begin{aligned}
	J_2 = - \frac{1}{2}s^5 \int_{0}^T \int_{\omega_0} \phi e^{-6s\vphi^* + 4s\widehat \vphi} \xi^9 \zeta_x \eta_{xx} - \frac{1}{2}s^5 \int_{0}^T \int_{\omega_0} \phi_x e^{-6s\vphi^* + 4s\widehat \vphi} \xi^9 \zeta \eta_{xx} ,
	\end{aligned}
\end{equation} 
and then by applying the  Young's inequality we have 
\begin{equation}\label{estimate-J-2}
	\begin{aligned}
		|J_2| \leq \epsilon s^3 \iint_{Q_T} e^{-2s\vphi} \xi^3 |\zeta_x|^2 +  \epsilon s^5 \iint_{Q_T} e^{-2s\vphi} \xi^5 |\zeta|^2 +
		C_\epsilon  s^7 \int_0^T \int_{\omega_0} e^{-12s\vphi^* + 10s\widehat \vphi} \xi^{15} |\eta_{xx}|^2 ,
	\end{aligned}
	\end{equation}
for given $\epsilon>0$.

\medskip 

(iii) {\underline{\em Estimate of $J_3$}}. Finally, the term $J_3$ can be estimated as follows: 
\begin{align}\label{estimate-J-3}
	|J_3| \leq  \epsilon s^5 \iint_{Q_T} e^{-2s\vphi} \xi^5 |\zeta|^2 + C_\epsilon s^5 \iint_{Q_T} e^{-12s\vphi^* + 10s\widehat \vphi} \xi^{13} |g_1|^2 . 
\end{align}

Combining   the estimates of $J_1$, $J_2$ and $J_3$ obtained in \eqref{estimate-J-1}, \eqref{estimate-J-2} and \eqref{estimate-J-3} (respectively), we acquire that the observation integral  \eqref{observation-zeta} of $\zeta$  satisfies:
\begin{equation}\label{esti-middle-zeta}
	\begin{aligned}
		&	s^5 \int_0^T \int_{\widehat \omega}  e^{-6s  \vphi^* + 4s\widehat \vphi  }    \xi^{9} |\zeta|^2  \\ 
	&	\leq  \epsilon s^5 \iint_{Q_T} e^{-2s\vphi} \xi^5 \left(3 |\zeta|^2 + |\eta|^2 \right)  + \epsilon s^3 \iint_{Q_T} e^{-2s\vphi} \xi^3 |\zeta_x|^2 + 2\epsilon s\iint_{Q_T} e^{-2s\vphi} \xi |\zeta_{xx}|^2 \\
	&  \quad  + C_\epsilon s^5  \iint_{Q_T} e^{-12s\vphi^* + 10s\widehat \vphi} \xi^{13} \left(|g_1|^2 + |g_3|^2 \right)   + 
	C_\epsilon s^9 \int_0^T \int_{\omega_0} e^{-12s\vphi^* + 10 s\widehat \vphi} \xi^{17}  |\eta|^2  \\
	& \quad  + C_\epsilon s^9 \int_0^T \int_{\omega_0} e^{-12s\vphi^* + 10 s\widehat \vphi} \xi^{17} \left( |\eta_x|^2 + |\eta_{xx}|^2  \right),
	\end{aligned}
\end{equation}  
for given $\epsilon>0$.

\medskip

$\bullet$ We apply similar process to handle the observation integral of $\theta$. By expressing 
\begin{align*}
	\theta = -\psi_t - \psi_{xxx} - g_2 \quad \text{in } \mathcal O \ \ (\text{consequently in } \omega_0) 
\end{align*}
from the equation \eqref{adj-extended-1}$_2$, and following the same steps as previous, one can achieve 
\begin{equation}\label{esti-middle-theta}
	\begin{aligned}
		&	s^5 \int_0^T \int_{\widehat \omega}  e^{-6s  \vphi^* + 4s\widehat \vphi  }    \xi^{9} |\theta|^2  \\ 
		&	\leq  \epsilon s^5 \iint_{Q_T} e^{-2s\vphi} \xi^5 \left(3 |\theta|^2 + |\psi|^2 \right)  + \epsilon s^3 \iint_{Q_T} e^{-2s\vphi} \xi^3 |\theta_x|^2 + 2\epsilon s\iint_{Q_T} e^{-2s\vphi} \xi |\theta_{xx}|^2 \\
		&  \quad  + C_\epsilon s^5  \iint_{Q_T} e^{-12s\vphi^* + 10s\widehat \vphi} \xi^{13} \left(|g_2|^2 + |g_4|^2 \right)   + 
		C_\epsilon s^9 \int_0^T \int_{\omega_0} e^{-12s\vphi^* + 10 s\widehat \vphi} \xi^{17}  |\psi|^2  \\
		& \quad  + C_\epsilon s^9 \int_0^T \int_{\omega_0} e^{-12s\vphi^* + 10 s\widehat \vphi} \xi^{17} \left( |\psi_x|^2 + |\psi_{xx}|^2  \right),
	\end{aligned}
\end{equation}  
for any $\epsilon>0$. 

\smallskip 

Fix $\epsilon>0$ small enough in \eqref{esti-middle-zeta} and  \eqref{esti-middle-theta}, so that the integrals  of $\zeta$, $\eta$, $\zeta_x$, $\zeta_{xx}$, $\theta$, $\psi$, $\theta_x$ and $\theta_{xx}$ in $Q_T$ can be absorbed by the associated leading integrals in the left hand side of \eqref{carleman-aux-2}. 

\medskip 

$\bullet$ Now, it remains to find  proper estimates for the observation integrals concerning  $\eta_x$, $\eta_{xx}$ in \eqref{esti-middle-zeta}, and $\psi_x$, $\psi_{xx}$ in \eqref{esti-middle-theta}. 
Indeed, writing the space $H^2(\omega_0)$ as an interpolation between $H^3(\omega_0)$ and $L^2(\omega_0)$,  we find that 
\begin{equation}\label{estimate-eta-x} 
	\begin{aligned}
	& s^9 \int_0^T \int_{\omega_0} e^{-12s\vphi^* + 10 s\widehat \vphi} \xi^{17} \left( |\eta_x|^2 + |\eta_{xx}|^2  \right) \\
	&\leq s^9 \int_0^T  e^{-12s\vphi^* + 10 s\widehat \vphi} \xi^{17} \|\eta\|^2_{H^2(\omega_0)} \\
	& \leq s^9 \int_0^T e^{-12s\vphi^* + 10 s\widehat \vphi} \xi^{17} \| \eta\|^{\frac{4}{3}}_{H^3(\omega_0)} \| \eta\|^{\frac{2}{3}}_{L^2(\omega_0)}  \\
	& \leq \varepsilon s \int_0^T e^{-2s\widehat \vphi} \xi^{-3} \|\eta\|^2_{H^3(0,L)} + C_\varepsilon s^{25} \int_0^T \int_{\omega_0} e^{-36s\vphi^* + 34s\widehat \vphi} \xi^{57} |\eta|^2 ,
	\end{aligned}
	\end{equation}
for any given  $\varepsilon>0$.

Similar technique will provide
\begin{equation}\label{estimate-psi-x} 
	\begin{aligned}
		& s^9 \int_0^T \int_{\omega_0} e^{-12s\vphi^* + 10 s\widehat \vphi} \xi^{17} \left( |\psi_x|^2 + |\psi_{xx}|^2  \right) \\
		& \leq \varepsilon s \int_0^T e^{-2s\widehat \vphi} \xi^{-3} \|\psi\|^2_{H^3(0,L)} + C_\varepsilon s^{25} \int_0^T \int_{\omega_0} e^{-36s\vphi^* + 34s\widehat \vphi} \xi^{57} |\psi|^2 .
 \end{aligned}
\end{equation}

Then, a proper choice of $\varepsilon>0$ helps to get rid of the first integrals in the right hand sides of \eqref{estimate-eta-x} and \eqref{estimate-psi-x} by means of the leading terms $s \int_0^T e^{-2s\widehat \vphi} \xi^{-3} \|\eta\|^2_{H^3(0,L)}$  and $s \int_0^T e^{-2s\widehat \vphi} \xi^{-3} \|\psi\|^2_{H^3(0,L)}$ in the left hand side of \eqref{carleman-aux-2}.

\smallskip

 As a consequence of the above analysis, the estimate \eqref{carleman-aux-2} boils down to the following:
 	\begin{equation}\label{carleman-aux-3}
 	\begin{aligned}
 		&	I(\eta,s) + I(\psi, s) + I(\zeta, s) + I(\theta, s) \\
 		& \leq  C s^3 \iint_{Q_T} e^{-2s\vphi^*} \xi \big(|g_{1,x}|^2 + |g_{2,x}|^2 + |g_{3,x}|^2 + |g_{4,x}|^2  \big) \\
& \ \ + C s^5 \int_{Q_T} e^{-12 s\vphi^* + 10 s\widehat \vphi} \xi^{13}\left(|g_1|^2 + |g_2|^2 + |g_3|^2 + |g_4|^2 \right) \\
& \ \  + C s^{25} \int_0^T \int_{\omega_0} e^{-36s\vphi^* + 34s\widehat \vphi} \xi^{57} |\eta|^2 + C s^{25} \int_0^T \int_{\omega_0} e^{-36s\vphi^* + 34s\widehat \vphi} \xi^{57} |\psi|^2 ,
 	\end{aligned}
 \end{equation}
 for all $s\geq \widehat s_0$. To this end, due to  the choices of $g_i \in L^2(0,T; H^1_0(0,L))$ for $i\in \{1,2,3,4\}$, and the definitions of $\vphi^*, \widehat \vphi$ given by \eqref{Weight-min-max}, we can easily conclude the required Carleman estimate \eqref{Carleman-main-adjoint}  for the adjoint system \eqref{adj-extended-1}--\eqref{adj-extended-2}. 
\end{proof}	
	
	\smallskip 
	
	\section{Null-controllability of the extended linearized system}\label{Section-local-null}

	 In this section, we establish the global null-controllability  of the extended linearized system \eqref{sys_linear-1}--\eqref{sys_linear-2}. As we have mentioned earlier, the main ingredient to prove the result is to first obtain a suitable observability inequality for the adjoint system \eqref{adj-extended-1}--\eqref{adj-extended-2}, and we do this in the following subsection.

	\subsection{Observability inequality}\label{Section-obser}  
	

	Let us first construct some modified Carleman weights (from the existing ones \eqref{Weight-1}--\eqref{Weight-2}) that do not vanish at $t=T$. We consider 
\begin{equation}\label{xi-new}
	\mathfrak{Z}(t)=
	\begin{dcases}
		\frac{1}{t(T-t)}, \ & 0 < t\leq  T/2, \\
		\frac{4}{T^2},  \ &T/2 < t\leq T,
	\end{dcases}
\end{equation}
	and 
	the weight function 
	\begin{align}\label{vphi-new}
		\mathfrak{S}(t,x) = \mathfrak{Z}(t) \beta(x) , \quad \forall (t,x) \in (0,T] \times [0,L] , 
	\end{align}
where the function $\beta$ is introduced  by \eqref{Weight-1}.

We further define 
\begin{equation}\label{weights-new} 
\begin{aligned}
&	\mathfrak{S}^*(t)	 =	\min_{[0,L]} 	\mathfrak{S}(t,x) = \mathfrak{Z}(t)\beta(l_{1/2}) , \qquad \forall t\in (0,T) \\
& 	\widehat{\mathfrak{S}}(t) =  \max_{[0,L]}	\mathfrak{S}(t,x)  =  \mathfrak{Z}(t) \big( \max\{ \beta(0), \beta(L)\} \big)  , \quad \forall t\in (0,T) .
\end{aligned}
\end{equation}  
	
	\begin{remark}\label{remark-weights-equivalent}
	Recall the former weight functins $\xi$, $\vphi$, and $\vphi^*$, $\widehat \vphi$ from \eqref{Weight-1}, \eqref{Weight-2} and \eqref{Weight-min-max} respectively. Then,  by construction we observe that 
		\begin{align*}
			\xi(t) = \mathfrak{Z}(t) , \ \ \text{in } (0,T/2]  , \quad \mathfrak{S}(t,x) = \vphi(t,x), \ \ \ \text{in } (0,T/2] \times [0,L] ,
			\end{align*}
	and further, 
	\begin{align*} 
			\mathfrak{S}^*(t) = \vphi^*(t) , \quad \widehat{\mathfrak{S}}(t) = \widehat \vphi(t), \quad \text{in } (0,T/2] . 
		\end{align*}
	\end{remark}

	With all these, we derive the following observability inequality associated to the adjoint system \eqref{adj-extended-1}--\eqref{adj-extended-2}. 
	
	\begin{proposition}\label{Prop-obs-ineq}
		Let $s$ be fixed in accordance with \Cref{Carleman-main}.  Then, there exists some constant $C>0$ that at most depends on $s$, $T$, $\omega$ and $\oo$  such that for any given source terms $g_i \in L^2(0,T; H^1_0(0,L))$ for $i=1,2,3,4$ and $(\zeta_0, \theta_0)\in [L^2(0,L)]^2$, the solution to \eqref{adj-extended-1}--\eqref{adj-extended-2} satisfies the   following estimate 
		\begin{equation}\label{Observ-main} 
			\begin{aligned}
				&	\|\zeta(T)\|^2_{L^2(0,L)} + \|\theta(T)\|^2_{L^2(0,L)} 
				+ \|e^{-s\widehat{\mathfrak{S}}} \mathfrak{Z}^{1/2} (\eta, \psi, \zeta, \theta) \|^2_{[L^\infty(0,T; L^2(0,L))]^4} 
				\\
		%
&	\quad + \iint_{Q_T} e^{-2 s\widehat{\mathfrak{S}}} \mathfrak{Z}^3 \left( |\eta_x|^2 + |\psi_x|^2 + |\zeta_x|^2 + |\theta_x|^2 
				\right) \\	
				&\leq 	 C  \iint_{Q_T} e^{-12s\mathfrak{S}^*+ 10s\widehat{\mathfrak{S}} } \mathfrak{Z}^{13} \big(|g_{1,x}|^2 + |g_{2,x}|^2 + |g_{3,x}|^2 + |g_{4,x}|^2  \big) \\
				& \quad    + C \int_0^T \int_{\omega_0} e^{-36s \mathfrak{S}^* + 34s \widehat{\mathfrak{S}}} \mathfrak{Z}^{57} \left(|\eta|^2 + |\psi|^2 \right). 
			\end{aligned}
		\end{equation}

	\end{proposition}  
	
	\begin{proof} The proof is made of two steps.

		\medskip 
		
$\bullet$	{\em \underline{Step 1.}}
	Let us recall  the definitions of weight functions $\vphi, \widehat \vphi$ from \eqref{Weight-2}, \eqref{Weight-min-max}.  Then, it is easy to observe from the Carleman estimate \eqref{Carleman-main-adjoint} that the following inequality holds true: 
	\begin{equation}\label{Carleman-modified}
		\begin{aligned}
			& \iint_{Q_T} e^{-2s\widehat \vphi} \xi^5 \left( |\eta|^2 + |\psi|^2 + |\zeta|^2 + |\theta|^2 
			\right) 
		+	\iint_{Q_T} e^{-2s\widehat \vphi} \xi^3 \left( |\eta_x|^2 + |\psi_x|^2 + |\zeta_x|^2 + |\theta_x|^2 
			  \right) \\
			& \leq   C  \iint_{Q_T} e^{-12s\vphi^*+ 10s\widehat \vphi} \xi^{13} \big(|g_{1,x}|^2 + |g_{2,x}|^2 + |g_{3,x}|^2 + |g_{4,x}|^2  \big) \\
			&   \   + C \int_0^T \int_{\omega_0} e^{-36s\vphi^* + 34s\widehat \vphi} \xi^{57} |\eta|^2 + C \int_0^T \int_{\omega_0} e^{-36s\vphi^* + 34s\widehat \vphi} \xi^{57} |\psi|^2 .
			\end{aligned} 
	\end{equation}
	
	Now, choose a function $\gamma \in \C^1([0,T])$ such that 
	\begin{align}\label{choice-gamma}
		\gamma = 0 \ \ \text{in } \, [0, T/4], \ \quad \gamma = 1 \ \ \text{in } \, [T/2, T] .
	\end{align}
It is clear that $\textnormal{Supp} (\gamma^\prime) \subset (T/4, T/2)$.

Consider the equations satisfied by $(\gamma \eta, \gamma \psi, \gamma \zeta , \gamma \theta)$, namely 
	\begin{align}\label{adj-extended-1-gamma}
	&\begin{cases}
		-(\gamma\eta)_t +\frac{1}{2}(\gamma\eta)_{xxx}  =  \gamma \zeta \mathds{1}_{\mathcal O} +  \gamma g_1 - \gamma^\prime \eta  &  \text{in } Q_T, \\
		-(\gamma \psi)_t - (\gamma \psi)_{xxx}   =  \gamma \theta \mathds{1}_{\mathcal O}  +  \gamma g_2 - \gamma^\prime \psi &  \text{in } Q_T, 
		\\
		(\gamma \eta)(t,0) = (\gamma \eta)(t,L) = (\gamma \eta_x)(t,L) = 0 & \text{for } t\in (0,T), \\
		(\gamma \psi)(t,0) = (\gamma \psi)(t,L) = (\gamma \psi_x)(t,0) =0 & \text{for } t\in (0,T) , \\
		(\gamma \eta)(T) = 0, \ \ (\gamma \psi)(T) = 0 & \text{in } (0,L) 
	\end{cases} \\
	\label{adj-extended-2-gamma}
	&\begin{cases} 	
		(\gamma\zeta)_t - \frac{1}{2}(\gamma \zeta)_{xxx} = \gamma g_3  + \gamma^\prime \zeta  &  \text{in } Q_T, \\
		(\gamma \theta)_t + (\gamma\theta)_{xxx}  = \gamma g_4  + \gamma^\prime \theta   &  \text{in } Q_T, \\
		(\gamma \zeta)(t,0) = (\gamma \zeta)(t,L) = (\gamma\zeta_x)(t,0) = 0 & \text{for } t\in (0,T), \\
		(\gamma \theta)(t,0) = (\gamma\theta)(t,L) = (\gamma \theta_x)(t,L) =0 & \text{for } t\in (0,T) , \\
		(\gamma\zeta)(0)=0, \ \ (\gamma\theta)(0)= 0 & \text{in } (0,L) , 
	\end{cases}
\end{align}
where we have used the adjoint equations  \eqref{adj-extended-1}--\eqref{adj-extended-2}. 

Applying \Cref{proposition-linear-well-adj} to   \eqref{adj-extended-1-gamma}--\eqref{adj-extended-2-gamma}
we get
\begin{equation}\label{energy-estimate-obs}
	\begin{aligned}
		&\| (\gamma \eta, \gamma \psi, \gamma \zeta , \gamma \theta) \|_{[L^2(0,T; H^1_0(0,L))]^4} + 	\| (\gamma \eta, \gamma \psi, \gamma \zeta , \gamma \theta) \|_{[L^\infty(0,T; L^2(0,L))]^4}
		\\
		&\leq C \left(\|(\gamma g_1, \gamma g_2, \gamma g_3, \gamma g_4)\|_{[L^2(0,T;L^2(0,L) )]^4} + \|(\gamma^\prime \eta, \gamma^\prime \psi, \gamma^\prime \zeta, \gamma^\prime \theta)\|_{[L^2(0,T; L^2(0,L))]^4}    \right) .
	\end{aligned}
\end{equation} 
	Thanks to the properties of $\gamma$ introduced in \eqref{choice-gamma}, we have  from \eqref{energy-estimate-obs}  that
	\begin{equation}\label{energy-estimate-obs-2}
		\begin{aligned}
			&\| (\eta, \psi, \zeta , \theta) \|_{[L^2(T/2,T; H^1_0(0,L))]^4} + 	\|\zeta(T)\|_{L^2(0,L)} + \|\theta(T)\|_{L^2(0,L)}
			\\
			&\leq C \left(\|(g_1,  g_2,  g_3,  g_4)\|_{[L^2(T/4,T;L^2(0,L) )]^4} + \|(\eta,  \psi,  \zeta,  \theta)\|_{[L^2(T/4,T/2; L^2(0,L))]^4}    \right) .
		\end{aligned}
	\end{equation}

 Note that the function  $e^{-2 s\widehat{\mathfrak{S}}} \mathfrak{Z}^5$ is bounded from  below in $[T/4,T/2]$ and therefore, in the right hand side of \eqref{energy-estimate-obs-2} we find that
\begin{equation}\label{ineq-T/2}
	\begin{aligned}
	&\|(\eta,  \psi,  \zeta,  \theta)\|^2_{[L^2(T/4,T/2; L^2(0,L))]^4} \\
	& \leq  C \int_{T/4}^{T/2}  \int_0^L e^{-2s\widehat{\mathfrak{S}}} \mathfrak{Z}^5 \left( |\eta|^2 + |\psi|^2 + |\zeta|^2 + |\theta|^2  \right) \\
	& \leq   C  \iint_{Q_T} e^{-12s\mathfrak{S}^*+ 10s\widehat{\mathfrak{S}} } \mathfrak{Z}^{13} \big(|g_{1,x}|^2 + |g_{2,x}|^2 + |g_{3,x}|^2 + |g_{4,x}|^2  \big) \\
	&   \quad  + C \int_0^T \int_{\omega_0} e^{-36s \mathfrak{S}^* + 34s \widehat{\mathfrak{S}}} \mathfrak{Z}^{57} \left(|\eta|^2 + |\psi|^2 \right),
   	\end{aligned}
	\end{equation}
	where we have used the fact that $\mathfrak{Z}=\xi$ and $\widehat {\mathfrak{S}} = \widehat \vphi$ in $[T/4,T/2]$ (see \Cref{remark-weights-equivalent}) as well as the Carleman estimate \eqref{Carleman-modified}. 
	
	\smallskip 
	
Also,	we can  incorporate the function  $e^{-2 s\widehat{\mathfrak{S}}} \mathfrak{Z}^n$ for any $n \in \mathbb N^*$  (which is bounded from above in $[T/2,T]$), in the estimate of $\|(\eta, \psi, \zeta, \theta)\|_{[L^2(T/2,T; H^1_0(0,L))]^4}$ in the left hand side of \eqref{energy-estimate-obs-2}.  This, combined with \eqref{ineq-T/2}, the estimate \eqref{energy-estimate-obs-2} follows
\begin{equation}\label{ineq-T/2-second} 
	\begin{aligned}
	&	\|\zeta(T)\|^2_{L^2(0,L)} + \|\theta(T)\|^2_{L^2(0,L)}      +  	\int_{T/2}^T \int_0^L e^{-2 s\widehat{\mathfrak{S}}} \mathfrak{Z}^5 \left( |\eta|^2 + |\psi|^2 + |\zeta|^2 + |\theta|^2 \right) \\
	& \quad + 
	\int_{T/2}^T \int_0^L e^{-2 s\widehat{\mathfrak{S}}} \mathfrak{Z}^3 \left( |\eta_x|^2 + |\psi_x|^2 + |\zeta_x|^2 + |\theta_x|^2 
	\right) \\	
&\leq 	 C  \iint_{Q_T} e^{-12s\mathfrak{S}^*+ 10s\widehat{\mathfrak{S}} } \mathfrak{Z}^{13} \big(|g_{1,x}|^2 + |g_{2,x}|^2 + |g_{3,x}|^2 + |g_{4,x}|^2  \big) \\
	& \quad    + C \int_0^T \int_{\omega_0} e^{-36s \mathfrak{S}^* + 34s \widehat{\mathfrak{S}}} \mathfrak{Z}^{57} \left(|\eta|^2 + |\psi|^2 \right),
	\end{aligned}
\end{equation}
 since $g_i \in L^2(0,T; H^1_0(0,L))$ for $i=1,2,3,4$.

 \medskip 
 
 On the other hand, since $\mathfrak{Z}=\xi$ and $\widehat {\mathfrak{S}} = \widehat \vphi$ in $(0,T/2]$ (\Cref{remark-weights-equivalent}), we deduce  that the 
  integrals
 \begin{align*}
 	&\int_{0}^{T/2}\int_0^L e^{-2 s\widehat{\mathfrak{S}}} \mathfrak{Z}^5 \left( |\eta|^2 + |\psi|^2 + |\zeta|^2 + |\theta|^2 
 \right) ,  \\ 
  \text{and } \
 	&\int_{0}^{T/2}\int_0^L e^{-2 s\widehat{\mathfrak{S}}} \mathfrak{Z}^3 \left( |\eta_x|^2 + |\psi_x|^2 + |\zeta_x|^2 + |\theta_x|^2 
 	\right)
 \end{align*} 
can be estimated by the same quantities appearing  in the right hand side of \eqref{ineq-T/2-second}  (thanks to the Carleman estimate \eqref{Carleman-modified}).  As a consequence, we have
\begin{equation}\label{ineq-0-T/2} 
	\begin{aligned}
		&	\|\zeta(T)\|^2_{L^2(0,L)} + \|\theta(T)\|^2_{L^2(0,L)} 
		+  
		\iint_{Q_T} e^{-2 s\widehat{\mathfrak{S}}} \mathfrak{Z}^5 \left( |\eta|^2 + |\psi|^2 + |\zeta|^2 + |\theta|^2 
		\right)
\\		
& \quad 		   +  
		\iint_{Q_T} e^{-2 s\widehat{\mathfrak{S}}} \mathfrak{Z}^3 \left( |\eta_x|^2 + |\psi_x|^2 + |\zeta_x|^2 + |\theta_x|^2 
		\right) \\	
		&\leq 	 C  \iint_{Q_T} e^{-12s\mathfrak{S}^*+ 10s\widehat{\mathfrak{S}} } \mathfrak{Z}^{13} \big(|g_{1,x}|^2 + |g_{2,x}|^2 + |g_{3,x}|^2 + |g_{4,x}|^2  \big) \\
		& \quad    + C \int_0^T \int_{\omega_0} e^{-36s \mathfrak{S}^* + 34s \widehat{\mathfrak{S}}} \mathfrak{Z}^{57} \left(|\eta|^2 + |\psi|^2 \right). 
	\end{aligned}
\end{equation}

\medskip 

$\bullet$ {\em \underline{Step 2.}} Let us define $\widehat \rho(t)= e^{-s\widehat{\mathfrak{S}}} \mathfrak{Z}^{1/2}$ so that $\widehat \rho(0)=0$. Again, by applying  \Cref{proposition-linear-well-adj} to the equations satisfied by 
$(\widehat \rho \eta, \widehat \rho \psi, \widehat \rho \zeta, \widehat \rho \theta)$,
	we get 
	\begin{equation}\label{energy-rho-adj}
		\begin{aligned}
	&\| (\widehat \rho \eta, \widehat \rho \psi, \widehat \rho \zeta , \widehat \rho \theta) \|_{[L^\infty(0,T; L^2(0,L))]^4}
			\\
		&	\leq C \left(\|(\widehat \rho g_1, \widehat \rho g_2, \widehat \rho g_3, \widehat \rho g_4)\|_{[L^2(0,T;L^2(0,L))]^4} + \|(\widehat \rho_t \eta, \widehat \rho_t \psi, \widehat \rho_t \zeta, \widehat \rho_t \theta)\|_{[L^2(0,T; L^2(0,L))]^4}    \right) .
			\end{aligned}
	\end{equation}
	Note that 
	\begin{align}
	 |\widehat \rho_t| \leq C s e^{-s\widehat{\mathfrak{S}}} \mathfrak{Z}^{5/2} \quad \text{since } |\widehat{\mathfrak{S}}_t| \leq C \mathfrak{Z}^2 ,
	\end{align}
for some constant $C>0$, 	and therefore, 
	\begin{equation}
		\begin{aligned}
	&	\|(\widehat \rho_t \eta, \widehat \rho_t \psi, \widehat \rho_t \zeta, \widehat \rho_t \theta)\|^2_{[L^2(0,T; L^2(0,L))]^4} \\
	&	\leq C \iint_{Q_T} e^{-2s\widehat{\mathfrak{S}}} \mathfrak{Z}^{5}\left(|\eta|^2 + |\psi|^2 + |\zeta|^2 + |\theta|^2  \right) \\ 
	&\leq 	 C  \iint_{Q_T} e^{-12s\mathfrak{S}^*+ 10s\widehat{\mathfrak{S}} } \mathfrak{Z}^{13} \big(|g_{1,x}|^2 + |g_{2,x}|^2 + |g_{3,x}|^2 + |g_{4,x}|^2  \big) \\
	& \ \    + C \int_0^T \int_{\omega_0} e^{-36s \mathfrak{S}^* + 34s \widehat{\mathfrak{S}}} \mathfrak{Z}^{57} \left(|\eta|^2 + |\psi|^2 \right). 	
		\end{aligned}
	\end{equation}
Using the above estimate in \eqref{energy-rho-adj}, and combining with \eqref{ineq-0-T/2}, we get the desired observability inequality \eqref{Observ-main}.
	
	The proof is complete.
	\end{proof}

	\subsection{Null-controllability}\label{Section-null-liniear}
	
	 This subsection is devoted to prove the global null-controllability of the extended linearized system \eqref{sys_linear-1}--\eqref{sys_linear-2} with initial data $(u_0,v_0)=(0,0)$ and source terms $f_i\in F$ with either  $F= L^1(0,T; L^2(0,L))$    or  $F= L^2(0,T; H^{-1}(0,L))$, for $i=1,2,3,4$. We mainly address the proof for $(f_1, f_2, f_3, f_4)\in [L^1(0,T; L^2(0,L))]^4$ and then, for the case when  $(f_1, f_2, f_3, f_4)\in [L^2(0,T; H^{-1}(0,L))]^4$,  we   point out the main changes in  the proof.


	\medskip

	Denote the  Banach space
	\begin{equation}\label{space_E}
\begin{aligned}
	\mathcal E := \Big\{ (u,v , p, q, h_1, h_2) \ | \ & e^{6s\mathfrak{S}^* - 5s \widehat{\mathfrak{S}} } {\mathfrak{Z}}^{-13/2}   (u, v, p, q) \in [L^2(0,T ; H^{-1}(0,L))]^4 ,  \\ 
	& e^{18s\mathfrak{S}^* - 17s \widehat{\mathfrak{S}} } {\mathfrak{Z}}^{-57/2} \big(h_1 \mathds{1}_\omega , h_2 \mathds{1}_\omega\big) \in [L^2((0,T)\times \omega)]^2 , \\  
	& e^{18s \mathfrak{S}^*  - 17 s \widehat{\mathfrak{S}} } \mathfrak{Z}^{-61/2} (u, v, p, q) \in [\C^0([0,T]; L^2(0,1))]^4  \\ 
	& \qquad \qquad \qquad \qquad \qquad \qquad \quad \cap [L^2(0,T; H^1_0(0,1))]^4  , \\ 
	& e^{s\widehat{\mathfrak{S}}} \mathfrak{Z}^{-1/2}\big(u_t - \frac{1}{2} u_{xxx} - h_1 \mathds{1}_{\omega}\big) \in L^1(0,T; L^2(0,L)), \\ 
	&  e^{s\widehat{\mathfrak{S}}} \mathfrak{Z}^{-1/2}\big(v_t + v_{xxx} - h_2 \mathds{1}_{\omega}\big) \in L^1(0,T; L^2(0,L)), \\ 
	& e^{s\widehat{\mathfrak{S}}} \mathfrak{Z}^{-1/2}\big(-p_t + \frac{1}{2} p_{xxx} - u \mathds{1}_{\oo}\big) \in L^1(0,T; L^2(0,L)) , \\                            
	&  e^{s\widehat{\mathfrak{S}}} \mathfrak{Z}^{-1/2}\big(-q_t - q_{xxx} - v \mathds{1}_{\oo}\big) \in L^1(0,T; L^2(0,L)), \\ 
	& \qquad p(T, \cdot)=q(T, \cdot)=0 \, \text{ in } (0,L) \Big\}, 
\end{aligned}
\end{equation}
and we prove the following null-controllability result. 	
	
\begin{proposition}\label{prop-null-cont-1}
	Let   $s$  be fixed parameter  according to \Cref{Carleman-main} and  $f_1, f_2, f_3, f_4$  be the functions satisfying 
	\begin{align}\label{condition-f_1-f_2} 
		e^{s\widehat{\mathfrak{S}}} \mathfrak{Z}^{-1/2} \big(f_1, f_2, f_3,f_4\big) \in [L^1(0,T; L^2(0,L))]^4  .
	\end{align} 
	Then, there exist controls $(h_1, h_2)$ and a solution $(u, v, p, q)$ to \eqref{sys_linear-1}--\eqref{sys_linear-2} in the space $\mathcal E$ such that we have 
	$p(0)=q(0)=0$ in $(0,L)$. 
\end{proposition}

\begin{proof}
We consider the space 
\begin{equation}  
\begin{aligned} 
	\mathcal Q_0 := \Big\{ (\eta, \psi, \zeta, \theta) \in [\C^4(\overline{Q_T})]^4 \, | \, & \,
		\eta(t,0) = \eta(t,L) = \eta_{x}(t,L) = 0 , \\
		& \Big(\frac{1}{2}\eta_{xxx} - \zeta \mathds{1}_{\oo}\Big)\Big|_{\{x=0\}} = \Big(\frac{1}{2}\eta_{xxx} - \zeta \mathds{1}_{\oo}\Big)\Big|_{\{x=L\}} =0    , \\
		&\psi(t,0) = \psi(t,L) = \psi_{x}(t,0)= 0 , \\ 
	 	& \left(\psi_{xxx} + \theta \mathds{1}_{\oo} \right)\big|_{\{x=0\}} = 	 \left(\psi_{xxx} + \theta \mathds{1}_{\oo} \right)\big|_{\{x=L\}}  =0   , \\
&	\zeta(t,0)=\zeta(t,L) = \zeta_{x}(t,0)= \zeta_{xxx}(t,0) = \zeta_{xxx}(t,L) =0 , \\
	&	\theta(t,0)=\theta(t,L) = \theta_{x}(t,L)= \theta_{xxx}(t,0) = \theta_{xxx}(t,L) =0       \Big\} ,	                 
	\end{aligned}
\end{equation}
and define the bi-linear form 
\begin{align}
	\mathcal L :  	\mathcal Q_0 \times 	\mathcal Q_0 \to \mathbb R ,
\end{align}
given  by 
\begin{equation}\label{bilinear-form} 
	\begin{aligned}
	&	\mathcal L\left(  (\eta, \psi, \zeta, \theta), (\underline{\eta}, \underline{\psi}, \underline{\zeta}, \underline{\theta} )     \right)  \\
	&	= \iint_{Q_T} e^{-12s\mathfrak{S}^*+ 10s\widehat{\mathfrak{S}} } \mathfrak{Z}^{13} \bigg[ \Big(-\eta_t + \frac{1}{2} \eta_{xxx} - \zeta \mathds{1}_{\oo}      \Big)_x \Big(-\underline{\eta}_t + \frac{1}{2} \underline{\eta}_{xxx} - \underline{\zeta} \mathds{1}_{\oo}      \Big)_x  \\
	& \qquad \qquad \qquad + \Big(-\psi_t - \psi_{xxx} - \theta \mathds{1}_{\oo}      \Big)_x \Big(-\underline{\psi}_t - \underline{\psi}_{xxx} - \underline{\theta} \mathds{1}_{\oo}      \Big)_x \\
	&  \qquad \qquad \qquad  +  \Big(\zeta_t - \frac{1}{2} \zeta_{xxx} \Big)_x \Big(\underline{\zeta}_t - \frac{1}{2} \underline{\zeta}_{xxx} \Big)_x 
	 + \Big(\theta_t + \theta_{xxx} \Big)_x \Big(\underline{\theta}_t + \underline{\theta}_{xxx} \Big)_x \bigg]  
	  \\
	& \ \  + \int_0^T \int_{\omega_0} e^{-36s\mathfrak{S}^*+ 34s\widehat{\mathfrak{S}} } \mathfrak{Z}^{57} \left( \eta \underline \eta + \psi \underline \psi \right) . 
	\end{aligned}
\end{equation}
We further define the linear operator $\ell : \mathcal Q_0 \to \mathbb R$ which is given by 
\begin{equation} 
\begin{aligned}\label{linear-form} 
&	\ell \left( (\eta, \psi, \zeta, \theta)  \right)  \\
	 & = \left\langle f_1, \eta \right \rangle_{L^1(L^2), L^\infty(L^2)} + \left\langle f_2, \psi \right \rangle_{L^1(L^2), L^\infty(L^2)} + 
	\left\langle f_3, \zeta \right \rangle_{L^1(L^2), L^\infty(L^2)} + \left\langle f_4, \theta \right \rangle_{L^1(L^2), L^\infty(L^2)}   .
\end{aligned}
\end{equation}

It is clear that  \eqref{bilinear-form}  defines an inner product since the observability inequality \eqref{Observ-main} holds. 
We denote by $\mathcal Q$, the closure of $\mathcal Q_0$ w.r.t. the norm $\mathcal L(\cdot,\cdot)^{1/2}$ and indeed it is an Hilbert space endowed with the inner product \eqref{bilinear-form}. The linear functional $\ell$ is also bounded;  
  in fact, we see
\begin{equation}
	\begin{aligned}
	& \left| \left\langle f_1, \eta \right \rangle_{L^1(L^2), L^\infty(L^2)} + \left\langle f_2, \psi \right \rangle_{L^1(L^2), L^\infty(L^2)} + 
	\left\langle f_3, \zeta \right \rangle_{L^1(L^2), L^\infty(L^2)} + \left\langle f_4, \theta \right \rangle_{L^1(L^2), L^\infty(L^2)} \right|  \\
 &\leq 
 \big\|e^{s\widehat{\mathfrak{S}}} \mathfrak{Z}^{-1/2} (f_1, f_2, f_3,f_4) \big\|_{[L^1(0,T; L^2(0,L))]^4}   \times 	\big\|e^{-s\widehat{\mathfrak{S}}} \mathfrak{Z}^{1/2} (\eta, \psi, \zeta,\theta) \big\|_{[L^\infty(0,T; L^2(0,L))]^4}  \\
  &	<  +\infty ,
	\end{aligned}
\end{equation}
which is possible due to the choice \eqref{condition-f_1-f_2} and the observation estimate \eqref{Observ-main}.

Therefore, by  Lax-Milgram's theorem,  there  exists  unique $(\widehat \eta, \widehat \psi, \widehat \zeta, \widehat \theta)\in \mathcal Q\times \mathcal Q$ which satisfies 
\begin{align*}
	\mathcal L \left( (\widehat \eta, \widehat \psi, \widehat \zeta, \widehat \theta), (\underline \eta, \underline \psi, \underline \zeta, \underline \theta  )  \right) = \ell \left(  (\underline \eta, \underline \psi, \underline \zeta, \underline \theta  )  \right) , \quad \forall    (\underline \eta, \underline \psi, \underline \zeta, \underline \theta  )  \in \mathcal Q. 
\end{align*}

Now, we set 
\begin{align}\label{solution-u}
	&\widehat u = e^{-12s\mathfrak{S}^*+10s\widehat{\mathfrak{S}}} \mathfrak{Z}^{13}\big(-\widehat \eta_t + \frac{1}{2}\widehat \eta_{xxx} - \widehat \zeta \mathds{1}_{\oo}  \big)_{xx}, \\ 	&\widehat v =  e^{-12s\mathfrak{S}^*+10s\widehat{\mathfrak{S}}} \mathfrak{Z}^{13}\big(-\widehat \psi_t -\widehat \psi_{xxx} - \widehat \theta \mathds{1}_{\oo}  \big)_{xx}, \label{solution-v}
	\\
	&	\widehat p =  e^{-12s\mathfrak{S}^*+10s\widehat{\mathfrak{S}}} \mathfrak{Z}^{13}\big(\widehat \zeta_t - \frac{1}{2}\widehat \zeta_{xxx} \big)_{xx},    \label{solution-p} \\
	&	\widehat q =  e^{-12s\mathfrak{S}^*+10s\widehat{\mathfrak{S}}} \mathfrak{Z}^{13}\big(\widehat \theta_t + \widehat \theta_{xxx}   \big)_{xx},  \label{solution-q} 
\end{align}
and 
\begin{align}\label{controls}
	\widehat h_1 = e^{-36 s \mathfrak{S}^* + 34 s\widehat{\mathfrak{S}} }{\mathfrak{Z}}^{57} \widehat \eta \mathds{1}_\omega , \quad  	\widehat h_2 =e^{-36 s \mathfrak{S}^* + 34 s\widehat{\mathfrak{S}} }{\mathfrak{Z}}^{57} \widehat \psi \mathds{1}_\omega .
\end{align}

Let us find the following bound for $\widehat u$; we have 
\begin{equation}\label{u-bound}
	\begin{aligned}
		&	\int_0^T e^{12s\mathfrak{S}^* - 10 s \widehat{\mathfrak{S}} } {\mathfrak{Z}}^{-13} \|\widehat u\|^2_{H^{-1}(0,L)} \\
		= &\int_0^T e^{12s\mathfrak{S}^* - 10 s \widehat{\mathfrak{S}} } {\mathfrak{Z}}^{-13} \sup_{ \|\vartheta\|_{H^1_0} =1 } \left|\left\langle \widehat u, \vartheta \right \rangle \right|^2_{H^{-1}, H^1_0} \\
		 =&
		\int_0^T e^{12s\mathfrak{S}^* - 10 s \widehat{\mathfrak{S}} } {\mathfrak{Z}}^{-13}
		\sup_{ \|\vartheta\|_{H^1_0} =1 }
		\left| \left\langle  e^{-12s\mathfrak{S}^* + 10 s \widehat{\mathfrak{S}} } {\mathfrak{Z}}^{13} \big( -\widehat \eta_t + \frac{1}{2} \widehat{\eta}_{xxx} - \widehat \zeta \mathds{1}_{\oo} \big)_{xx}
		   , \vartheta \right \rangle \right|^2_{H^{-1}, H^1_0} \\
	 \leq & \iint_{Q_T} e^{-12s\mathfrak{S}^* + 10 s \widehat{\mathfrak{S}} } {\mathfrak{Z}}^{13}
		\left|\big( -\widehat \eta_t + \frac{1}{2} \widehat{\eta}_{xxx} - \widehat \zeta \mathds{1}_{\oo} \big)_{x} \right|^2 \\
		\leq &\mathcal L\big( (\widehat \eta, \widehat \psi, \widehat \zeta, \widehat \theta) ,  (\widehat \eta, \widehat \psi, \widehat \zeta, \widehat \theta) \big)  < +\infty . 
	\end{aligned}
\end{equation}  

In a similar way, we can find the required bounds for $\widehat v$, $\widehat p$ and $\widehat q$. Also, in a straightforward way we can show the boundedness of the control functions $(\widehat h_1, \widehat h_2)$. 

Eventually, we have the following bound:
\begin{equation}\label{Bound-sol-control}
	\begin{aligned} 
	&\big\| e^{6s\mathfrak{S}^* - 5 s \widehat{\mathfrak{S}} } {\mathfrak{Z}}^{-13/2} (\widehat u ,  \widehat v, \widehat p, \widehat q)  \big\|_{[L^2(0,T; H^{-1}(0,L))]^4}
	\\  & \quad \quad + \big\|e^{18s\mathfrak{S}^* - 17 s \widehat{\mathfrak{S}} } {\mathfrak{Z}}^{-57/2} (\widehat h_1, \widehat h_2)  \big\|_{[L^2((0,T)\times \omega)]^2} < +\infty  ,
	 \end{aligned}
\end{equation}
and this $(\widehat u, \widehat v, \widehat p, \widehat q)$ is the unique solution to the linearized system \eqref{sys_linear-1}--\eqref{sys_linear-2}  in the sense of transposition with the control functions $\widehat h_1$ and $\widehat h_2$.  Moreover, by construction of solutions \eqref{solution-p} and \eqref{solution-q}, it is clear that 
$$\widehat p(0)=0, \quad \widehat q(0)=0 \quad \text{in } (0,L),$$
which is the required null-controllability result for our system \eqref{sys_linear-1}--\eqref{sys_linear-2}.

\bigskip

Let us now define
\begin{align}\label{states-new}
	(u^*, v^*, p^*, q^*): = e^{18s \mathfrak{S}^*  - 17 s \widehat{\mathfrak{S}} } \mathfrak{Z}^{-61/2} ( \widehat u, \widehat v, \widehat p, \widehat q ) ,
\end{align}
so that it satisfies $(u^*(0), v^*(0), p^*(0), q^*(0)  )= (0,0,0,0)$. Indeed, we observe that (using the expression of $\widehat p$ from  \eqref{solution-p}) 
\begin{align*}
	p^* &= e^{18s \mathfrak{S}^*  - 17 s \widehat{\mathfrak{S}} } \mathfrak{Z}^{-61/2} e^{-12s\mathfrak{S}^*+10s\widehat{\mathfrak{S}}} \mathfrak{Z}^{13}\big(\widehat \zeta_t - \frac{1}{2}\widehat \zeta_{xxx} \big)_{xx}\\
	&= e^{ 6 s \mathfrak{S}^* - 7 s \widehat{\mathfrak{S}}  } \mathfrak{Z}^{-35/2}  \big(\widehat \zeta_t  -  \frac{1}{2}\widehat \zeta_{xxx}  \big)_{xx},
\end{align*}
so that by definitions of weight functions $\mathfrak{S}^*$, $\widehat{\mathfrak{S}}$ given by  \eqref{weights-new}, it is clear that $p^*(0)=0$. In a similar manner, we can show that this phenomenon holds for the  functions  $u^*$, $v^*$ and $q^*$ in \eqref{states-new}.  

\smallskip 

Let us now write the equations satisfied by $(u^*, v^*, p^*, q^*)$,  which are
		\begin{align}
	\label{sys_linear-sec-con-1}
	&\begin{cases}
		u^*_t -\frac{1}{2} u^*_{xxx}  = h^*_1 \mathds{1}_\omega + f^*_1  +  \big( e^{18s \mathfrak{S}^*  - 17 s \widehat{\mathfrak{S}} } \mathfrak{Z}^{-61/2}  \big)_t \widehat u    &  \text{in } Q_T, \\
		v^*_t + v^*_{xxx}   = h^*_2 \mathds{1}_\omega + f^*_2  +  \big(e^{18s \mathfrak{S}^*  - 17 s \widehat{\mathfrak{S}} } \mathfrak{Z}^{-61/2}\big)_t \widehat v    &  \text{in } Q_T, 
		\\
		u^*(t,0) = u^*(t,L) = u^*_x(t,0) = 0 & \text{for } t\in (0,T), \\
		v^*(t,0) = v^*(t,L) = v^*_x(t,L) =0 & \text{for } t\in (0,T) , \\
		u^*(0) = 0, \ \ v^*(0) = 0 & \text{in } (0,L)  ,
	\end{cases} \\
	\label{sys_linear-sec-con-2}
	&\begin{cases} 	
		-p^*_t + \frac{1}{2}p^*_{xxx} =  u^* \mathds{1}_{\mathcal O} + f^*_3  -  \big( e^{18s \mathfrak{S}^*  - 17 s \widehat{\mathfrak{S}} } \mathfrak{Z}^{-61/2}  \big)_t \widehat p   &  \text{in } Q_T, \\
		-q^*_t - q^*_{xxx}  =  v^* \mathds{1}_{\mathcal O}  + f^*_4   -  \big( e^{18s \mathfrak{S}^*  - 17 s \widehat{\mathfrak{S}} } \mathfrak{Z}^{- 61/2}  \big)_t \widehat q  &  \text{in } Q_T, \\
		p^*(t,0) = p^*(t,L) = p^*_x(t,L) = 0 & \text{for } t\in (0,T), \\
		q^*(t,0) = q^*(t,L) = q^*_x(t,0) =0 & \text{for } t\in (0,T) , \\
		p^*(T)=0, \ \ q^*(T)=0 & \text{in } (0,L),
	\end{cases}
\end{align}
where $
(h^*_1, h^*_2) :=  e^{18s \mathfrak{S}^*  - 17 s \widehat{\mathfrak{S}} } \mathfrak{Z}^{-61/2} (\widehat h_1 , \widehat h_2)$,
and they belong to $[L^2((0,T)\times \omega)]^2$, thanks to \eqref{Bound-sol-control}.  

Also, we have 
\begin{align*}
	(f^*_1, f^*_2, f_3^*, f_4^*): =  e^{18s \mathfrak{S}^*  - 17 s \widehat{\mathfrak{S}} } \mathfrak{Z}^{-61/2} (f_1 ,f_2, f_3, f_4) \in [L^1(0,T; L^2(0,L))]^4 ,
\end{align*}
since $e^{18s \mathfrak{S}^*  - 17 s \widehat{\mathfrak{S}} } \mathfrak{Z}^{-61/2} \leq C e^{s\widehat{\mathfrak{S}}} \mathfrak{Z}^{-1/2}$.

Beside the above, one can  compute that 
\begin{align*}
\left|	\big( e^{18s \mathfrak{S}^*  - 17 s \widehat{\mathfrak{S}} } \mathfrak{Z}^{-61/2}  \big)_t \right| \leq C  e^{18s \mathfrak{S}^*  - 17 s \widehat{\mathfrak{S}} } \mathfrak{Z}^{-57/2} .
	\end{align*} 
Consequenty, 
\begin{align*}
\left|	\big( e^{18s \mathfrak{S}^*  - 17 s \widehat{\mathfrak{S}} } \mathfrak{Z}^{-61/2}  \big)_t \widehat u \right|
&\leq C \left| e^{18s \mathfrak{S}^*  - 17 s \widehat{\mathfrak{S}} } \mathfrak{Z}^{-57/2} \widehat u \right| \\
& \leq  C e^{12s(\mathfrak{S}^* -  \widehat{\mathfrak{S}}) } \mathfrak{Z}^{-44/2} \left|  e^{6s \mathfrak{S}^*  -  5 s \widehat{\mathfrak{S}}  } \mathfrak{Z}^{-13/2} \widehat u   \right|  ,
\end{align*}
and then  by using the fact $\mathfrak{S}^* \leq \widehat{\mathfrak{S}}$ and    \eqref{Bound-sol-control},  we deduce  $\big( e^{18s \mathfrak{S}^*  - 17 s \widehat{\mathfrak{S}} } \mathfrak{Z}^{-61/2}  \big)_t \widehat u \in L^2(0,T; H^{-1}(0,L))$. In a similar way, we can show that the same phenomenon holds true  for the related terms in the right hand side of the equations satisfied by $v^*, p^*, q^*$.

\smallskip 
Altogether, we have shown  that each source term in the set of equations \eqref{sys_linear-sec-con-1}--\eqref{sys_linear-sec-con-2} belongs to the space $L^2(0,T; H^{-1}(0,L))$. As a result, by applying  \Cref{proposition-linear-well}, we have 
\begin{align*}
	(u^*, v^*, p^*, q^*) \in   [\C^0([0,T]; L^2(0,L))]^4 \cap [L^2(0,T; H^{1}_0(0,L) )]^4 , 
\end{align*} 
and moreover,
\begin{align*}
\|(u^*, v^*, p^*, q^*)\|_{ [\C^0([0,T]; L^2(0,L))]^4 \cap [L^2(0,T; H^{1}_0(0,L) )]^4 }  < +\infty .
\end{align*}
Therefore, the functions $(\widehat u, \widehat v, \widehat p, \widehat q, \widehat h_1, \widehat h_2)$ belong to the space $\mathcal E$ defined  by \eqref{space_E}. 

The proof is complete.  
\end{proof}

\paragraph{\bf Analogous null-controllability result}

Recall that, we have proved the controllability result in \Cref{prop-null-cont-1} with the choices  $f_i \in L^1(0,T; L^2(0,L))$ for $i=1,2,3,4$ verifying \eqref{condition-f_1-f_2}. We can     derive    a similar control result when $f_i \in L^2(0,T; H^{-1}(0,L))$ with few changes in the proof.  

\smallskip 

More specifically, we now consider a Banach space 
`` $\widetilde{\mathcal E}$ "   
that contains the elements  $(u, v, p, q, h_1, h_2)$ that verify the {\em first three conditions} and {\em last condition} of \eqref{space_E}, with in addition 
\begin{align*}
& e^{s\widehat{\mathfrak{S}}} \mathfrak{Z}^{-3/2}\big(u_t - \frac{1}{2} u_{xxx} - h_1 \mathds{1}_{\omega}\big) \in L^2(0,T; H^{-1}(0,L)), \\
&  e^{s\widehat{\mathfrak{S}}} \mathfrak{Z}^{-3/2}\big(v_t + v_{xxx} - h_2 \mathds{1}_{\omega}\big) \in L^2(0,T; H^{-1}(0,L)), \\ 
& e^{s\widehat{\mathfrak{S}}} \mathfrak{Z}^{-3/2}\big(-p_t + \frac{1}{2} p_{xxx} - u \mathds{1}_{\oo}\big) \in L^2(0,T; H^{-1}(0,L)) , \\                            
&  e^{s\widehat{\mathfrak{S}}} \mathfrak{Z}^{-3/2}\big(-q_t - q_{xxx} - v \mathds{1}_{\oo}\big) \in L^2(0,T; H^{-1}(0,L)) .
\end{align*}


\medskip

In this case, we have the following result. 

\begin{proposition}\label{prop-null-cont-2}
	Let   $s$  be fixed parameter  according to \Cref{Carleman-main} and  $f_1, f_2, f_3, f_4$  be the functions satisfying 
	\begin{align}\label{condition-f_1-f_2-2} 
		e^{s\widehat{\mathfrak{S}}} \mathfrak{Z}^{-3/2} \big(f_1, f_2, f_3,f_4\big) \in [L^2(0,T; H^{-1}(0,L))]^4  .
	\end{align} 
	Then, there exists controls $(h_1, h_2)$ and a solution $(u, v, p, q)$ to \eqref{sys_linear-1}--\eqref{sys_linear-2} in the space $\widetilde{\mathcal E}$ such that we have 
	$p(0)=q(0)=0$ in $(0,L)$. 
\end{proposition}

\begin{proof}
	The proof follows almost the same lines as the proof for \Cref{prop-null-cont-1} except we now consider 
	the linear functional $\ell : \mathcal Q_0 \to \mathbb R$  as 
	\begin{equation} 
		\begin{aligned}\label{linear-form-2} 
			&	\ell \left( (\eta, \psi, \zeta, \theta)  \right)  \\
			=&  \left\langle f_1, \eta \right \rangle_{L^2(H^{-1}), L^2(H^1_0)} + \left\langle f_2, \psi \right \rangle_{L^2(H^{-1}), L^2(H^1_0)} + 
			\left\langle f_3, \zeta \right \rangle_{L^2(H^{-1}), L^2(H^1_0)} + \left\langle f_4, \theta \right \rangle_{L^2(H^{-1}), L^2(H^1_0)}   .
		\end{aligned}
	\end{equation}
This verifies 
\begin{equation}
	\begin{aligned}
	&\left|  \left\langle f_1, \eta \right \rangle_{L^2(H^{-1}), L^2(H^1_0)} + \left\langle f_2, \psi \right \rangle_{L^2(H^{-1}), L^2(H^1_0)} + 
		\left\langle f_3, \zeta \right \rangle_{L^2(H^{-1}), L^2(H^1_0)} + \left\langle f_4, \theta \right \rangle_{L^2(H^{-1}), L^2(H^1_0)}     \right| \\
&	\leq 
	\big\|e^{s\widehat{\mathfrak{S}}} \mathfrak{Z}^{-3/2} (f_1, f_2, f_3,f_4) \big\|_{[L^2(0,T; H^{-1}(0,L))]^4}   \times 	\big\|e^{-s\widehat{\mathfrak{S}}} \mathfrak{Z}^{3/2} (\eta, \psi, \zeta,\theta) \big\|_{[L^2(0,T; H^1_0(0,L))]^4} \\
&	\leq 
	\big\|e^{s\widehat{\mathfrak{S}}} \mathfrak{Z}^{-3/2} (f_1, f_2, f_3,f_4) \big\|_{[L^2(0,T; H^{-1}(0,L))]^4}   \times 	\big\|e^{-s\widehat{\mathfrak{S}}} \mathfrak{Z}^{3/2} (\eta_x, \psi_x, \zeta_x,\theta_x) \big\|_{[L^2(0,T; L^2(0,L))]^4} \\
& <   +\infty ,	 
	\end{aligned}
\end{equation}
	thanks to the choices of $f_i$  ($i=1,2,3,4$) in \eqref{condition-f_1-f_2-2} and the observation inequality \eqref{Observ-main}.

	\smallskip 
	
We skip the other   details of the proof since those are similar with the proof of \Cref{prop-null-cont-1}. 
\end{proof}

\medskip

\section{Local null-controllability of the extended  nonlinear  system}\label{Section-null-nonlinear}

In this section, we  prove the main theorem of our paper, 
 that is \Cref{thm:main} and as explained in Section \ref{Sec-Introduction}, this is equivalent to prove the local null-controllability of the extended system \eqref{sys_equiv_1}--\eqref{sys_equiv_2}, which is precisely \Cref{thm:main_extended}.  
To prove it, we use  the following well-known result. 

\begin{theorem}[Inverse mapping theorem]\label{thm-inverse}
	Let $\mathcal G_1$, $\mathcal G_2$ be two Banach spaces and $\Y: \G_1 \to \G_2$ be a map  satisfying $\Y\in \C^1(\G_1; \G_2)$. Assume that $b_1\in \G_1$,  $\Y(b_1)=b_2\in \G_2$ and $\Y^\prime(b_1):\G_1\to \G_2$ is surjective. Then, there exists $\delta>0$ such that for every $\widetilde b\in \G_2$ satisfying $\|\widetilde b - b_2\|_{\G_2}<\delta$, there exists a solution of the equation
	$$\Y(b)=\widetilde b, \quad b\in \G_1. $$
\end{theorem}

\smallskip


\begin{proof}[\bf Proof of \Cref{thm:main_extended}]
We apply \Cref{thm-inverse} to prove the required local null-controllability result for the system \eqref{sys_equiv_1}--\eqref{sys_equiv_2}. We take 
\begin{align*}
	\G_1  = \mathcal E , \quad \G_2 = [\mathcal F]^4 ,
\end{align*}
where $\mathcal E$ is defined by \eqref{space_E} and 
\begin{align}\label{space-F}
	\mathcal F : = \Big\{ f  \mid  e^{s  \widehat{\mathfrak{S}} } \mathfrak{Z}^{-1/2} f \in L^1(0,T; L^2(0,L))       \Big\} .
\end{align}
Now, define the map $\mathcal Y : \G_1 \to \G_2$ such that 
\begin{align}\label{map-Y}
	\Y\left(u,v,p,q,h_1,h_2\right)   
	&= \Big( u_t - \frac{1}{2} u_{xxx} -  3 uu_x + 6vv_x - h_1 \mathds{1}_\omega , \, v_t + v_{xxx} + 3uv_x - h_2\mathds{1}_\omega, \\
	& \qquad \qquad  -p_t + \frac{1}{2} p_{xxx} - 3pu_x + 3qv_x - u\mathds{1}_\oo , \, -q_t -q_{xxx} +  6pv_x - v\mathds{1}_\oo \Big) .  \notag 
\end{align}

\begin{itemize}
	\item Let us first check  that $\Y \in \C^1(\G_1; \G_2)$. In this regard, we denote 
	the space 
	\begin{align}
	\mathcal P : = \Big\{  y  \mid  e^{18s \mathfrak{S}^* - 17 s\widehat{\mathfrak{S}} } \mathfrak{Z}^{-61/2} y \in L^2(0,T; H^1_0(0,L))      \Big\} .
	\end{align}
	Then observe that, to prove $\Y \in \C^1(\G_1; \G_2)$,  
  it is enough to show that the map 
	\begin{align}\label{map-2}
		(y,z) \in [ \mathcal P ]^2  \mapsto yz_x \in  \mathcal F 
	\end{align}
	is continuous. 
	
	\smallskip 
	
	Recall the construction of weight functions $\mathfrak{S}^*$, $\widehat{\mathfrak{S}}$ in \eqref{weights-new}. Then,  as we have obtained \eqref{cond-weights-max-min}, one has   
	\begin{align*}
     36 s  \mathfrak{S}^* (t) - 35 s  \widehat{\mathfrak{S}} (t) \geq c_0 s \mathfrak{Z}(t) , 
	\end{align*}
for all $t\in (0,T]$ and for some $c_0>0$. 
Consequently, 
	\begin{align}\label{cond-non}
		e^{ s \widehat{\mathfrak{S}} } \mathfrak{Z}^{-1/2} \leq  e^{36 s  \mathfrak{S}^*- 34  \widehat{\mathfrak{S}}} \mathfrak{Z}^{-61} \times e^{-  c_0 s \mathfrak{Z}(t) } \mathfrak{Z}^{ 61-\frac{1}{2} } \leq C   e^{36 s  \mathfrak{S}^*- 34  \widehat{\mathfrak{S}}} \mathfrak{Z}^{-61}  
	\end{align}
for some constant $C>0$. 

\smallskip 

	Using \eqref{cond-non}, we now compute the following:  for any two functions $y, z \in  \mathcal P$, 
	\begin{equation}\label{bound-nonlinear}
	\begin{aligned}
		&\| y z_x \|_{\mathcal F}  = \int_0^T  e^{s\widehat{\mathfrak{S}}} \mathfrak{Z}^{-1/2} \|yz_x \|_{L^2(0,L)} \\
& \quad 	\leq   C \int_0^T e^{18 s  \mathfrak{S}^*- 17 \widehat{\mathfrak{S}}} \mathfrak{Z}^{-61/2} \|y\|_{L^\infty(0,L)} \,  e^{18 s  \mathfrak{S}^*- 17 \widehat{\mathfrak{S}}} \mathfrak{Z}^{-61/2} \|z_x\|_{L^2(0,L)} 
\\
& \quad 	 \leq 
	C \int_0^T e^{18 s  \mathfrak{S}^*- 17 \widehat{\mathfrak{S}}} \mathfrak{Z}^{-61/2} \|y\|_{H^1_0(0,L)} \, e^{18 s  \mathfrak{S}^*- 17 \widehat{\mathfrak{S}}} \mathfrak{Z}^{-61/2} \|z\|_{H^1_0(0,L)} 
	\\
& \quad   \leq  C \big\|e^{18 s  \mathfrak{S}^*- 17 \widehat{\mathfrak{S}}} \mathfrak{Z}^{-61/2} y   \big\|_{L^2(0,T; H^1_0(0,L))} \big\|  e^{18 s  \mathfrak{S}^*- 17 \widehat{\mathfrak{S}}} \mathfrak{Z}^{-61/2} z \big\|_{L^2(0,T; H^1_0(0,L))},
	\end{aligned} 
	\end{equation}
and thus the  continuity of the map \eqref{map-2} follows. 

Once we have \eqref{bound-nonlinear}, it is not difficult to conclude that $\Y \in \C^1(\G_1; \G_2)$.

\medskip 

\item Next, we check that $\Y^\prime(0,0,0,0,0)$ is surjective. In fact, we have 
$\Y(0,0,0,0,0) = (0, 0, 0, 0)$, and $\Y^\prime(0,0,0,0,0) : \G_1 \to G_2$ is given by 
	\begin{align*}
	&\Y^\prime(0,0,0,0,0)(u,v,p,q,h_1,h_2) \\
	&= \Big( u_t - \frac{1}{2} u_{xxx} - h_1 \mathds{1}_\omega , \, v_t + v_{xxx}  - h_2\mathds{1}_\omega, \,  -p_t + \frac{1}{2} p_{xxx}   - u\mathds{1}_\oo , \, -q_t -q_{xxx} - v\mathds{1}_\oo \Big) ,
	\end{align*}
which is surjective due to the controllability result given by \Cref{prop-null-cont-1}.

\smallskip

Set   $b_1=(0,0,0,0,0)$,  $b_2 =(0,0,0,0)$ and consider 
$\widetilde b=(\xi_1, \xi_2,0,0)\in \G_2$, where $(\xi_1,\xi_2)$ is the given external source term  in \eqref{sys_equiv_1}--\eqref{sys_equiv_2} or in \eqref{Control-system-main}.  Then, according to \Cref{thm-inverse},   
there is a $\delta>0$ such that for given $(\xi_1, \xi_2)$ verifying 
\begin{align*}
	\left\|(\xi_1, \xi_2,0,0)\right\|_{\G_2} < \delta ,
\end{align*} 
 there exists a solution-control pair  $(u,v,p,q,h_1,h_2)\in \G_1=\mathcal E$ to the system \eqref{sys_equiv_1}--\eqref{sys_equiv_2}. In particular, $p(0)=q(0)=0$ in $(0,L)$. This completes the proof of \Cref{thm:main_extended} which implies the proof for \Cref{thm:main}.  
\end{itemize}
\end{proof}

\begin{remark}
	The above local null-controllability result can also be derived by  considering the space
\begin{align*}  
	\mathcal F  = \Big\{ f  \mid  e^{s  \widehat{\mathfrak{S}} } \mathfrak{Z}^{-3/2} f \in L^2(0,T; H^{-1}(0,L))       \Big\} 
\end{align*}
	 instead of \eqref{space-F}. 	
\end{remark}

\section{Concluding remarks}\label{Section-conclusion}

In this work, we have proved the existence of insensitizing controls for a system of two KdV equations, notably known as Hirota-Satsuma system.
 The insensitizing problem is reformulated to  an equivalent null-control problem for the $4\times 4$ system 
 \eqref{sys_equiv_1}--\eqref{sys_equiv_2} with the action of  {\em only two localized controls}. 
 To prove this result, we first established a  Carleman estimate (see \Cref{Carleman-main}) for the adjoint system \eqref{adj-extended-1}--\eqref{adj-extended-2} to the linearized problem \eqref{sys_linear-1}--\eqref{sys_linear-2}. This helped us to find  a suitable observability inequality, namely \eqref{Observ-main},  which lead to prove the global null-controllability of the linearized model \eqref{sys_linear-1}--\eqref{sys_linear-2}. Finally, using the inverse mapping theorem, we conclude the local null-controllability of the system   \eqref{sys_equiv_1}--\eqref{sys_equiv_2}, and that implies the main result of this paper, i.e.,  \Cref{thm:main}.
 
 \medskip

  Let us now  present some concluding remarks concerning the problem addressed in this work.

	\begin{itemize} 
		\item[1.] {\em Choice of initial data.}  As in other insensitizing problems, the assumption on the zero  initial condition  in \Cref{thm:main} (i.e., $u_0=v_0=0$) is roughly related to the fact that system \eqref{sys_equiv_1}--\eqref{sys_equiv_2} is composed by forward and backward equations.
		In fact, it has already been indicated  in the work \cite{deTZ09} (see also \cite{deT00}) that it is   difficult to treat every  initial data while studying the insensitizing control problems. 
		

		
		\smallskip 
		
		\item[2.] {\em Condition on observation region.} The assumption $\mathcal O \cap \omega \neq \emptyset$ is essential to prove a suitable Carleman estimate and hence the  observability inequality (see \Cref{Carleman-main} and \Cref{Prop-obs-ineq} in this paper), which are the main ingredients for the proof of \Cref{thm:main_extended}. However, in   \cite{KdeT10} the authors  proved that in the simpler case of heat equation, this condition is not necessary if one deals with an {\em $\varepsilon$-insensitizing problem} (that is, $\Big|\frac{\partial J_\tau(u,v)}{\partial \tau}\big|_{\tau=0}\Big|\leq \varepsilon$). This remains  an open problem for our case.  
		
		\smallskip

	\item[3.] {\em Less observation term in sentinel functional.} Recall that, we have considered the following sentinel functional in this paper, 
	\begin{align}\label{senti-con}
		J_\tau (u,v) = \frac{1}{2} \iint_{(0,T)\times \oo} |u|^2 + \frac{1}{2} \iint_{(0,T)\times \oo} |v|^2 .
	\end{align} 
	If we drop any one of the two observations, then the extended linearized system \eqref{sys_linear-1}--\eqref{sys_linear-2}  will not be controllable anymore by the methodology used in this paper.   For instance, by choosing the sentinel functional 
	\begin{align*}
		\widehat J_\tau (u,v) = \frac{1}{2} \iint_{(0,T)\times \oo} |u|^2 
	\end{align*}
	would lead the extended system as
	\begin{align*}
		\begin{cases}
		u_t -\frac{1}{2} u_{xxx} - 3uu_x + 6vv_x = h_1 \mathds{1}_\omega + \xi_1 \ &\text{in } Q_T, \\
		v_t + v_{xxx} + 3uv_x = h_2 \mathds{1}_\omega + \xi_2 \ &\text{in } Q_T, \\
		-p_t + \frac{1}{2} p_{xxx} - 3pu_x + 3qv_x = u\mathds{1}_\oo \ &\text{in } Q_T, \\
		-q_t +  q_{xxx} + 6pv_x = 0 \ &\text{in } Q_T,
		\end{cases}
	\end{align*}
		with the same initial-boundary conditions as in \eqref{sys_equiv_1}--\eqref{sys_equiv_2}. But the associated linearized model to the above system is not null-controllable with two controls $(h_1, h_2)$  due to the lack of linear coupling in the equation of $q$. Thus, the choice of \eqref{senti-con} is somewhat necessary to study the insensitizing control property for our system \eqref{Control-system-main}, at least by means of the strategy developed in this paper.

				\smallskip 
				
\item[4.] {\em Reduction of control  input.}  It would be a challenging task  if we reduce the number of controls in the system \eqref{Control-system-main}. As for instance, if we drop the action of $h_2$ in the second equation of \eqref{Control-system-main}, it is quite clear that the extended linearized system \eqref{sys_linear-1}--\eqref{sys_linear-2}  cannot be controllable since there is no influence of control force in the equations of $v$ and $q$.   
 %
 %
 %
 Therefore, dealing with only one control instead of two in the main system \eqref{Control-system-main} is really delicate and it needs further attention. Generally speaking,  for the purpose of studying insensitizing control problems,  we need to tackle with the controllability of an extended system where the number of controls are already less than the equations (in our case, we have two controls $(h_1,h_2)$ acting in the $4\times 4$ system \eqref{sys_equiv_1}--\eqref{sys_equiv_2}). 
 Thus, reducing  the number of controls in the main system would lead more obstacle to achieve the required controllability criterion for the extended system.
 Similar phenomenon has  also been addressed in \cite{Bhandari2023insensitizing} in the context of insensitizing problem for a coupled fourth- and second-order parabolic system, and   in \cite{HSdT18} in the framework of hierarchic control problems.

		
	\end{itemize}

\section*{Acknowledgement} 
\begin{center} 
This  work  is partially supported by the Czech-Korean project GA\v{C}R/22-08633J.
\end{center}

	\bigskip

	\bibliographystyle{plain}

	\bibliography{ref_ins_ks}
	
\end{document}